\documentclass[11pt]{article}

\usepackage{fullpage}
\usepackage{slashbox}
\usepackage{hyperref}
\usepackage{amsmath}
\usepackage{amsfonts}

\newcommand{\R}{{\mathbb{R}}}
\newcommand{\N}{{\mathbb{N}}}
\renewcommand{\vec}{\mathbf}
\newcommand{\BigO}{\mathcal O}
\newcommand{\eps}{\varepsilon}

\newtheorem{theorem}{Theorem}[section]
\newtheorem{remark}[theorem]{Remark}
\newtheorem{lemma}[theorem]{Lemma}

\newtheorem{definition}[theorem]{Definition}
\newenvironment{proof}[1][Proof]{\textbf{#1.} }{~ \hfill \rule{0.5em}{0.5em}\\ \indent}

\title{Derivation of the Isotropic Diffusion Source Approximation (IDSA) for Supernova Neutrino Transport by Asymptotic Expansions
\thanks{This work was funded by the HP2C Initiative (High Performance and High Productivity Computing) in the framework of the project ``Stellar Explosions''.}
}

\author{H.~Berninger\footnotemark[2]
\and E.~Fr\'enod\footnotemark[3]
\and M.~Gander\footnotemark[2]
\and M.~Liebend\"orfer\footnotemark[4]
\and J.~Michaud\footnotemark[2]
}

\begin{document}

\maketitle

\renewcommand{\thefootnote}{\fnsymbol{footnote}}

\footnotetext[2]{Universit\'e de Gen\`eve, Section de Math\'ematiques, 2-4,~rue du Li\`evre, CP~64, CH-1211~Gen\`eve ({\tt Heiko.Berninger|Martin.Gander|Jerome.Michaud@unige.ch})}
\footnotetext[3]{Universit\'e de Bretagne-Sud, Laboratoire de Math\'ematiques, Centre Yves Coppens, Bat.~B, BP~573, \mbox{F-56017~Vannes} ({\tt emmanuel.frenod@univ-ubs.fr})}
\footnotetext[4]{Universit\"at Basel, Departement Physik, Klingelbergstrasse~82, CH-4056~Basel ({\tt Matthias.Liebendoerfer@unibas.ch})}

\renewcommand{\thefootnote}{\arabic{footnote}}

~\\
{\small {\bf Abstract - }
We present Chapman--Enskog and Hilbert expansions applied to the $\BigO(v/c)$ Boltzmann equation for the radiative transfer of neutrinos in core-collapse supernovae. Based on the Legendre expansion of the scattering kernel for the collision integral truncated after the second term, we derive the diffusion limit for the Boltzmann equation by truncation of Chapman--Enskog or Hilbert expansions with reaction and collision scaling. We also give asymptotically sharp results obtained by the use of an additional time scaling. The diffusion limit determines the diffusion source in the \emph{Isotropic Diffusion Source Approximation (IDSA)} of Boltzmann's equation \cite{LiebendoerferEtAl09big,BeFrGaLiMiVa12_esaimbig} for which the free streaming limit and the reaction limit serve as limiters. Here, we derive the reaction limit as well as the free streaming limit by truncation of Chapman--Enskog or Hilbert expansions using reaction and collision scaling as well as time scaling, respectively. Finally, we motivate why limiters are a good choice for the definition of the source term in the IDSA.
}
\\~\\
{\small {\bf Keywords - } 
Boltzmann equation, radiative transfer, neutrino, core-collapse supernova, asymptotic expansion, diffusion limit
}
\\~\\
{\small {\bf AMS subject classification - }
35B40,
35Q20,
35Q85,
82C70,
85-08, 
85A25
}

\thispagestyle{plain}

\section*{Introduction}

Some astrophysical processes like nuclear burning in the sun, neutron
star mergers, or the collapse and explosion of a massive star emit
enormeous quantities of neutrinos from captures of electrons or
positrons on nucleons. If such events occur in our Galaxy, the
neutrinos can be recorded on Earth together with the electromagnetic
spectrum of the corresponding astrophysical events. However, since
the neutrinos only weakly interact with matter, they are more
difficult to detect than photons. On the other hand, due to their
weak interactions, they can escape from much denser and hotter matter
than the photons so that they carry information about this extreme
state of matter directly to the terrestrial observer. Moreover,
they are not obscured by dust and may even be detectable from events
that are not visible in the electromagnetic spectrum. {By the
detection of neutrinos, models describing nuclear burning in the
sun were confirmed. Later, neutrino oscillations were detected
that lead to extensions of the standard model of particle physics~\cite{BahcallSears72,AhmadEtAl02}. Furthermore, 
the general core-collapse scenario of type II supernovae
was confirmed by the detection of neutrinos from supernova 1987A,
i.e., the explosion of a massive star in the close-by
Large Magellanic Cloud~\cite{BiontaEtAl87,HirataEtAl87}.}

Astrophysicists started to appreciate the role of neutrinos in the
explosion of massive stars in 1966 already~\cite{ColgateWhite66big,Arnett66}, only 10 years
after the postulated neutrino has been detected as a free particle~\cite{CowanReinesEtAl56}.
Ever since, simulations of core-collapse supernovae attempted
to take the drain of energy and lepton number by the emission of
neutrinos into account~\cite{BetheWilson85,Bruenn85,MyraBludman89,WilsonMayle93}.
In fact, more than 99\% of the
released gravitational binding energy of the neutron star that
forms at the center of the stellar collapse is carried away by
neutrinos~\cite{Bethe90}.
It turned out that neutrinos are not
only important as messengers from density and temperature regimes
that are not otherwise accessible by observation, but they also
have a large impact on the dynamics of the supernova itself. The
neutrino energy is considered crucial for reviving and feeding the
shock that finally leads to the explosion of the star. But the exact
mechanism of this process is still not completely understood. The
main issue is that sufficiently accurate and reliable radiative
transfer methods were initially only implemented in spherical
symmetry, as for example the solution of the 
$\BigO(v/c)$ Boltzmann equation~\cite{MezzacappaBruenn93b,RamppJanka00,ThompsonBurrowsPinto03}
or the complete general relativistic Boltzmann equation~{\cite{LiebendoerferEtAl01,LiebendoerferEtAl04big,SumiyoshiEtAl06}.}
But the explosion dynamics
and the interaction of the neutrinos with the convective stellar
matter and magnetic fields at the core of the explosion are
fundamentally three-dimensional processes~\cite{HerantEtAl94,JankaMueller96}.
It is therefore crucial
to investigate the reliability of efficient approximations of the
radiative transfer equations that can more efficiently be
used in three-dimensional models. 
For instance, one seeks approximations of the Boltzmann equation that capture the main
processes of neutrino transport in different regimes while being computationally more
economic.

{In principle, methods for the solution of the non-relativistic Boltzmann equation can be extended to the relativistic Boltzmann equation. For instance, Monte Carlo techniques used in BAMPS \cite{XuGreiner05} or the relativistic lattice Boltzmann (RLB) algorithm \cite{HuppEtAl11, RomatschkeMendosaSucci11} have been successfully applied to high energy physics systems. With regard to approximations of the relativistic Boltzmann equation the same is true. For example, one can linearize the collision term and get so-called model equations. In the non-relativistic setting this leads to the BGK model which was formulated independently by Bhatnagar, Gross and Krook \cite{BGK54} and Welander~\cite{Welander54}. In a relativistic framework, Marle \cite{Marle65, Marle69a, Marle69b} and Anderson and Witting \cite{AndersonWitting74} developed analogues of the BGK model for relativistic fluids. In contrast to the former model, the latter permits particles with zero rest mass and thus can be used for models in which neutrinos are treated as massless. The applicability of the Anderson--Witting model to the analysis of systems that could describe magnetic white dwarfs or cosmological fluids has been demonstrated in~\cite[p.\;218]{CercignaniKremer02}. 
In all core-collpase supernova
models that involve the Boltzmann equation or an approximation of it, the right hand side of the equation always contains terms for the emission and absorption of neutrinos by matter. Since these terms can be interpreted as emission out of and absorption into an equilibrium distribution, they represent exacly the structure of the right hand side of the model equations. In this respect, Boltzmann equations used in core-collpase supernova simulations can always be regarded as generalizations of the model equations.
To further simplify the models and reduce the cost of the algorithms, one can also approximate the transport operator on the left hand side of the Boltzmann equation. 
Well-known approximations of the transport operator such as the
Multi-Group Flux-Limited Diffusion approximations~\cite{OttEtAl08,BruennEtAl10}
or the Variable Eddington Factor Method~\cite{MarekJanka09,MuellerEtAl10,RamppJanka02}
have successfully been extended to axi-symmetric simulations of supernovae. In this paper we investigate an approximation of the transport operator.}

The main physical behaviour of neutrinos in core-collapse supernovae is that they are trapped by the matter in high density regimes, where their dynamics is governed by diffusion, whereas they are practically freely streaming in low density regimes further away from the core. In both cases, the Boltzmann equation can be reduced considerably. This is the underlying idea of the decomposition used in the \emph{Isotropic Diffusion Source Approximation (IDSA)} of Boltzmann's equation~\cite{LiebendoerferEtAl09big}. Concretely, the distribution function of the neutrinos is decomposed additively into a trapped and a streaming particle component on the whole domain. The resulting equations that govern these two particle components are reduced to the main physical properties of the components. The major challenge of the approximation is to find an appropriate coupling between the two reduced equations. In the IDSA they are supposed to be coupled by a source term. The concrete form of this source term, the diffusion source, is based on the diffusion limit and, for non-diffusive regimes, limited from above and below on the basis of the free streaming and reaction limits. The derivation of the diffusion limit in terms of an asymptotic expansion is sketched in~\cite[App.\;A]{LiebendoerferEtAl09big} and given in full detail in~\cite{BeFrGaLiMiVa12_esaimbig}.

The main purpose of this paper is to derive the diffusion, free streaming and reaction limits in the IDSA by Chapman--Enskog and Hilbert expansions of variations of the reaction and collision scaled as well as the time scaled Boltzmann equation. We give the concrete order of the approximations to the Boltzmann equation in terms of the scaling parameter $\eps>0$ and thus provide a deeper understanding of the approximation quality for radiative neutrino transfer given by the IDSA. Here, we concentrate almost entirely on the analysis of the IDSA. Computational aspects, discretization and numerical solution techniques for the IDSA and the Boltzmann equation accompanied with numerical results which compare the two models are presented in~\cite{LiebendoerferEtAl09big,BeFrGaLiMiVa12_esaimbig}. 

Concretely, the main topics and results discussed in this paper are as follows. In Section~\ref{ModelsForCCSN}, we recall two models used for the simulation of core-collapse supernovae, namely a fully coupled model of hydrodynamic and $\BigO(v/c)$ Boltzmann equations and the Isotropic Diffusion Source Approximation. 
We restrict our considerations to spherical symmetry although extensions to the full 3D case, in particular to an IDSA model in 3D, are feasible~\cite[Sec.\;4]{LiebendoerferEtAl09big}. With regard to the IDSA, we explain how the hypotheses on the trapped and streaming particle components lead to the corresponding reduced equations. We perform a Legendre expansion of the scattering kernel that is crucial for the derivation of the diffusion and the reaction limit of Boltzmann's equation. The three regimes of the IDSA represented by the diffusion source and its limiters are discussed mathematically and physically. Finally, we address some mathematical issues of the IDSA.

Section~\ref{asymptotics}, which contains the asymptotic analysis, is the core of this paper. Here, we perform a Chapman--Enskog expansion of Boltzmann's equation and then consider Hilbert expansions. The results can be summarized as follows. Introducing a scaling parameter $\eps>0$ that accounts for large reaction and collision terms, we can derive the diffusion limit of the Boltzmann equation up to the order $\BigO(\eps^2)$ if we neglect the same terms that needed to be neglected in the ``leading order'' approximation in~\cite{LiebendoerferEtAl09big,BeFrGaLiMiVa12_esaimbig}. With the same scaling and without the need to neglect these terms, the reaction limit of the Boltzmann equation up to the order $\BigO(\eps)$ can be derived. Here, the limit used in the IDSA is obtained if the free streaming component decreases strongly enough asymptotically. 

In principle, the same results can be derived by Chapman--Enskog expansions as well as by Hilbert expansions. However, the Hilbert expansion approach seems the more natural one since in fact, it is finally used for the actual comparison with Boltzmann's equation. For the application of Chapman--Enskog expansions, we resort to a lemma that provides approximations of Chapman--Enskog expansions by Hilbert expansions. This lemma requires additional assumptions on the asymptotic behaviour of components of the expanded distribution function in the Chapman--Enskog expansion that are thus not needed if we start with the Hilbert expansion. Nevertheless, we find that these connections between Chapman--Enskog and Hilbert expansions give a good insight into the nature of these expansions that are interesting on their own.

As the main sharp asymptotic result, we provide a derivation of the diffusion limit of Boltzmann's equation and thus of the diffusion source up to the order $\BigO(\eps^2)$ with the use of an additional time scaling by which it is no longer necessary to neglect any further terms. Finally, the free streaming limit of the Boltzmann equation is derived up to the order $\BigO(\eps)$ if only time scaling of the equation is considered. The paper closes with a motivation why it is natural to use the free streaming and the reaction limits as limiters of the diffusion source as it is done in the IDSA.


\section{Models for Core-Collapse Supernovae}
\label{ModelsForCCSN}

In this section, we introduce two models for the simulation of core-collapse supernovae. The first model consists of a coupled system of hydrodynamic and Boltzmann equations. The second model is the Isotropic Diffusion Source Approximation (IDSA) of the first model and has been developed in Liebend\"orfer et~al.~\cite{LiebendoerferEtAl09big}. Although the two models can be considered both in spherical symmetry and without symmetry in 3D, we restrict ourselves to the spherically symmetric case here. The description of the models resembles in parts the presentation given in \cite{BeFrGaLiMiVa12_esaimbig}, however, the discussion of them is more comprehensive here.

\subsection{Fully coupled model}
\label{fully_coupled}

A widely accepted model for the simulation of core-collapse supernovae is a coupled system of the hydrodynamic equations \eqref{Hydro} for the background matter (including leptonic number conservation) with Poisson's equation~\eqref{Grav} for gravity and Boltzmann's equation \eqref{Boltz} for the radiative transfer of neutrinos. We refer to papers of different groups~\cite{Bruenn85,MezzacappaBruenn93abig}, \cite{RamppJanka02}, \cite{ThompsonBurrowsPinto03} and \cite{SumiyoshiEtAl06} who have implemented different algorithms for the same physics and, in particular, to the comparison paper~\cite{LiebendoerferEtAl05}. However, all these approaches are restricted to spherically symmetric cases.
The presentation in this paper is based on the version used in~\cite{LiebendoerferEtAl09big}. We write the system of equations in short as
\begin{eqnarray}
\frac{\partial \vec u}{\partial t} + \nabla\cdot \vec F(\vec u) + \nabla\vec p(\vec u)&=& \vec G(\vec u,\phi)+\vec S(\vec u,f) \label{Hydro}\\[2mm]
\Delta \phi&=& 4\pi G \rho \label{Grav}\\[2mm]
\underbrace{\frac{1}{c}\frac{d f}{d t}+ \mu \frac{\partial f}{\partial r} +F_\mu(\vec u)\frac{\partial f}{\partial \mu}+F_\omega(\vec u)\frac{\partial f}{\partial \omega}}_{\mathcal D(f)} &=& j(\vec u) \underbrace{-\tilde\chi(\vec u) f+  \mathcal C(\vec u,f)}_{\mathcal J(f)}\,. \label{Boltz}
\end{eqnarray}
This system is given in spherical symmetry, i.e., for $(r,t)\in\Omega_T:=(0,R)\times (0,T)$ for some radius $R>0$ and some end time $T>0$. Here, $(0,R)$ represents the spatial domain $\Omega\subset\R^3$, which is the ball around the origin with radius~$R$, where the matter of the star that needs to be considered is thought to be contained. In the following we specify the equations in more detail.

\subsubsection{Hydrodynamics, state equation and gravity}
\label{hydrostategrav}

System \eqref{Hydro} contains the hydrodynamic equations for an ideal fluid, i.e., the Euler equations \cite{LandauLifshitz87}, as well as a balance law for the leptonic number. The state vector $\vec u\in\R^4$, which is a function on~$\Omega_T$, the nonlinear flux function $\vec F:\R^4\to\R^4$ and the source terms $\vec G=\vec G(\vec u,\phi)$ and $\vec S=\vec S(\vec u,f)$ have the form
  \begin{equation*}
  \vec u =\begin{pmatrix}
   \rho\\
    \rho v\\ 
     E \\
      \rho Y_e 
      \end{pmatrix}
 \,, \quad\,
  \vec F =\begin{pmatrix}
\rho  v\\
  v\rho v\\
v\left(E+p(\vec u)\right)\\
v\rho Y_e  
\end{pmatrix}\,, \quad\,
 \vec G=\begin{pmatrix}
0\\
-\rho \nabla \phi\\
- \rho v\nabla\phi \\
0
\end{pmatrix}\,, \quad\,
\vec S= \begin{pmatrix}
0\\
S_v(\vec u,f)\\
S_E\left(\vec u,f\right)\\
S_{Y_e}\left(\vec u,f\right)
\end{pmatrix}\,,
\end{equation*}
while, as usual, the gradient of the term $\vec p(\vec u)=(0,p(\vec u),0,0)^T$ with the pressure $p(\vec u)$ only enters the momentum equation. The variables contained in the state vector $\vec u$ are the density field $\rho$ and the velocity $v$ in radial direction of the background matter. More precisely, since the baryon number is the conserved quantity in astrophysical hydrodynamics, whether relativistic or classical, $\rho=n_b m_b$ is considered to be the product of the baryon number density $n_b$ and the baryonic mass constant $m_b$. Furthermore, the total energy density $E$ of the matter is given as the sum of its internal energy density $\rho e$ and its kinetic energy density $\rho\frac{v^2}{2}$. The pressure $p$ is assumed to be given by an equation of state $p=p(\rho,\vartheta(e),Y_e)$ 
where $\vartheta$ is the temperature (depending on the internal specific energy~$e$) and $Y_e=\frac{n_e}{n_b}$ is the net electron fraction of the matter \cite[p.\;1179]{LiebendoerferEtAl09big}. Here, $n_e$ is the electron minus the positron number density, so that $\rho Y_e=n_e m_b$ is the net electron density up to the constant $m_b$ \cite[pp.\;640/1]{MezzacappaBruenn93abig}. 
The appearance of $Y_e$ in the equation of state (in~\cite{LiebendoerferEtAl09big} the Lattimer--Swesty equation~\cite{LattimerSwesty91} is used) takes the compositional degrees of freedom into account. Under the extreme densities in the stellar core and in case of nuclear statistical equilibrium that we consider, the pressure is essentially given by the degeneracy pressure of the electrons. 
This dependency of the pressure on $Y_e$ necessitates to add a fourth balance law (leptonic number conservation) to the Euler equations which describes the evolution of~$Y_e$. 

On the right hand side, the bilinear term $\vec G(\vec u,\phi)$ accounts for the effect of the gravitational field of the matter onto itself, see, e.g.,~\cite[Sec.\;7]{TracPen03}. The Newtonian gravitational potential $\phi$ on $\Omega$ is given by the Poisson equation \eqref{Grav} where $G$ denotes the gravitational constant \cite[p.\;6]{LandauLifshitz87}. With regard to equations \eqref{Hydro} and \eqref{Grav} we recall that in spherical symmetry, with the radial unit vector~$\mathbf{e}_r$, the gradient and the divergence operators are given by 
\begin{equation}
\label{nabla_div_in_r}
\nabla\psi(r)=\frac{\partial }{\partial r}\psi(r)\,\mathbf{e}_r\qquad\mbox{and}\qquad\nabla\cdot \vec F(r)=\frac{1}{r^2}\frac{\partial }{\partial r}\left(r^2 F_r(r)\right)\,,
\end{equation}
for scalar fields $\psi:\R\to\R$ and vector fields $\vec F=(F_r,0,0)^T:\R\to\R^3$ (see, e.g., \cite[p.\;679]{Mihalas84}). 

The coupling of system \eqref{Hydro} for the matter with the Boltzmann equation \eqref{Boltz} for the neutrinos is provided by the source term $\vec S(\vec u,f)$ which adds to global momentum, energy and leptonic number conservation \eqref{Hydro}$_2$--\eqref{Hydro}$_4$. Generally speaking, $\vec S(\vec u,f)$ consists of sums of certain energy and angular moments of the right hand side of~\eqref{Boltz} and is nonlinear in $\vec u$ and, due to the restricted set of neutrino interactions with matter in Section~\ref{radiative}, linear in $f$. For instance, if the neutrinos were isotropic, $S_v(\vec u,f)$ would become the negative gradient of the neutrino pressure which corresponds to the pressure gradient of the matter on the left hand side of~\eqref{Hydro}, compare~\cite[Eq.~(24)]{LiebendoerferEtAl09big}. For further details we refer to \cite{MezzacappaBruenn93abig} and~\cite{Michaud13}. 

\begin{remark}
\label{supernova}
With regard to supernova simulations, it is crucial and, for our purpose, enough to know that the right hand side $\vec S(\vec u,f)$ of the hydrodynamic equations does not depend explicitly on~$f$ but only on energy and angular moments of $f$. In fact, the explosion of a star can be seen by the velocity of the background matter and can be inferred by the detection of neutrinos~\cite{Woosley89}. However, only the neutrino flux and energy spectrum are observable, but not their angular distribution $f$. It~therefore suffices to consider only moments of $f$ instead of $f$ itself. This is done both in the IDSA and in the asymptotics that we perform for the radiative transfer equations in Section~\ref{asymptotics}.
\end{remark}

Both hydrodynamics \eqref{Hydro} and gravitation \eqref{Grav} could be replaced by more accurate general relativistic versions~{\cite{LiebendoerferEtAl04big,RamppJanka02,MarekEtAl06,Kappeli11}.}  For instance, the gravitational binding energy released in a core-collapse supernova is equivalent to about $1/10$ of the mass of the remaining neutron star. This fact had already been claimed by W.~Baade and F.~Zwicky two years after the discovery of the neutron~\cite{BaadeZwicky34}. Recall that the mass conservation equation \eqref{Hydro}$_1$ is in fact a baryon number conservation equation, so that the mass defect is not ignored here. Since the gravitational binding energy is almost completely converted into neutrino energy \cite{RamppJanka02}, the energy balance equation \eqref{Hydro}$_3$ can be regarded as the conservation of energy including gravitational binding energy. The supernova models considered here are also valid if only a protoneutron star forms at the core and later collapses into a black hole due to continued accretion before, during or after the explosion, when part of the material does not reach escape velocity and falls back.

\subsubsection{Radiative transfer}
\label{radiative}

The radiative transfer equation used in \cite{LiebendoerferEtAl09big} for the transport of neutrinos is the $\BigO(v/c)$ Boltzmann equation~\eqref{Boltz} in spherical symmetry. It represents both the special and the general relativistic transport equation for massless fermions moving with the speed of light $c$ up to the order $\BigO(v/c)$ where $v$ is the velocity of the background matter. 
We refer to \cite[Sec.~95]{Mihalas84} for a derivation and \cite{RamppJanka02} for a discussion of possibly additional $\BigO(v/c)$ terms that might have to be considered. In practice, the fraction $v/c$ can be expected to be around $1/10$ in most of the computational domain except for the shock region where $v/c\approx 3/10$ can occur and at very high densities deep inside the neutron star where it gets very close to~$1$, see~\cite{LiebendoerferEtAl04big}. In that sense the $\BigO(v/c)$ Boltzmann equation must be regarded as an approximative model to the physical reality. {Apart from the relativistic aspect contained in the left hand side, equation~\eqref{Boltz} also takes the quantum aspect of the fermionic particles into account. This is achieved by introducing blocking factors in the emission and collision terms on the right hand side, cf.~\cite{UehlingUhlenbeck33} or~\cite{CercignaniKremer02}. Details will be given below}. 

The Boltzmann equation \eqref{Boltz} describes the evolution of the neutrino distribution function 
\begin{equation}
\label{fDef}
f:[0,T]\times (0,R]\times [-1,1]\times (0,E]\to [0,1]\,,\quad (t,r,\mu,\omega)\mapsto f(t,r,\mu,\omega)\,.
\end{equation}
Here, $\mu\in [-1,1]$ is the cosine of the angle between the outward radial direction and the direction of neutrino propagation and $\omega\in (0,E]$ with some maximal $E>0$ is the neutrino energy. In the formulation of \eqref{Boltz} that we use, both phase space variables $\mu$ and $\omega$ are given in the frame comoving with the background matter, cf.~\cite[p.\;640]{MezzacappaBruenn93abig}. In more general approaches, one could replace $(0,E]$ by $(0,\infty)$. 

With regard to the notation we use the Lagrangian time derivative 
\begin{equation}
\label{Lagrange}
\frac{d}{d t} = \frac{\partial}{\partial t} + v\frac{\partial}{\partial r}
\end{equation} 
in the comoving frame. Furthermore, the coefficient function $F_\mu(\vec u)$ weighting the $\mu$-derivative in~\eqref{Boltz} can be decomposed as 
\begin{equation*}
F_\mu(\vec u) = F_\mu^0 + F_\mu^1(\vec u)\,,\quad F_\mu^0 =\frac{1}{r}(1-\mu^2)\,,\quad F_\mu^1(\vec u)=\mu \left( \frac{d \ln\rho}{cd t}+\frac{3v}{cr}\right)(1-\mu^2)\,,
\end{equation*}
where the term $F_\mu^0\frac{\partial f}{\partial \mu}$ accounts for the change in propagation direction due to inward or outward movement of the neutrino. The second term $F_\mu^1(\vec u)\frac{\partial f}{\partial \mu}$ represents the angular aberration (i.e., the change observed in the comoving frame) of the neutrino propagation direction due to the motion of the fluid. Finally, the product of the coefficient function
\begin{equation*}
F_\omega(\vec u) =\left[\mu^2 \left( \frac{d \ln\rho}{cd t}+\frac{3v}{cr}\right)-\frac{v}{cr} \right]\omega
\end{equation*}
with $\frac{\partial f}{\partial \omega}$ in \eqref{Boltz} accounts for the (Doppler--)shift in neutrino energy due to the motion of the matter. The left hand side of the Boltzmann equation is abbreviated by $\mathcal{D}(f)$ where $\mathcal{D}$ is a linear operator. 


{The dependency of $\mathcal{D}$ on the hydrodynamic variables in $\vec u$ occurring in the comoving frame vanishes in case of a static background with frozen matter. In this case one can pass to the laboratory frame by setting $F_\mu^1(\vec u)=F_\omega(\vec u)=0$ and thus obtain the equation
\begin{equation}
\label{Boltz_frozen}
\frac{1}{c}\frac{\partial f}{\partial t}+ \mu \frac{\partial f}{\partial r} +\frac{1}{r}(1-\mu^2)\frac{\partial f}{\partial \mu} 
= j{}-\tilde\chi{} f+  \mathcal C(f)\,,
\end{equation}
in which Lorentz transformed quantities are used, see \cite[\S 90,\;\S 95]{Mihalas84}.}


Although in the infall phase \cite[p.\;787]{Bruenn85} it is enough to consider only electron neutrinos $\nu_e$, for postbounce simulations \cite[p.\;1179]{LiebendoerferEtAl09big} one needs at least two Boltzmann equations in order to obtain the transport of both electron neutrinos $\nu_e$ and electron antineutrinos $\bar\nu_e$. In general, one needs to include muon and tau neutrinos and their antiparticles, too~{\cite{BurasEtAl03big, Keil03}.} With our set of interactions given below, all different types of neutrinos are transported independently, so that one has to deal with up to six Boltzmann equations that are all coupled with the hydrodynamic equations by the source term $S(\vec u,f)$. Since the basic structure of these equations is the same, it is enough for our purpose to consider only one Boltzmann equation \eqref{Boltz} as a prototype. We point out, however, that these equations are coupled if we include interactions between different neutrino types like neutrino pair reactions~\cite[pp.\;774/5]{Bruenn85} or, as a future perspective, neutrino flavour oscillations~{\cite{BahcallSears72,AhmadEtAl02,Keil03, LentzEtAl12}.}

As for system \eqref{Hydro}, the coupling with hydrodynamics is provided by source terms on the right hand side of~\eqref{Boltz}. Their values usually differ for different neutrino flavours and their antiparticles. The interaction of neutrinos with matter includes emission and absorption
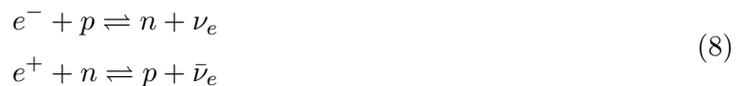
\begin{equation}
\label{reactions}
\begin{aligned}
e^-+p\rightleftharpoons n+\nu_e\\[1mm]
e^++n\rightleftharpoons p+\bar\nu_e
\end{aligned}
\end{equation}
of electron neutrinos $\nu_e$ or electron antineutrinos $\bar\nu_e$ by protons $p$ or neutrons $n$, the forward reactions being known as electron $e^-$ or positron $e^+$ capture. Analogous reactions occur in case of electron $e^-$ or positron $e^+$ capture by nuclei. They depend on the state of the background matter and result in neutrino emissivity $j(\omega,\vec u)$ and absorptivity $\chi(\omega,\vec u)$ whose sum is the {stimulated absorptivity 
} $\tilde\chi=j+\chi$ in \eqref{Boltz}, cf.~\cite[p.\;639]{MezzacappaBruenn93abig} and \cite[p.\;1177]{LiebendoerferEtAl09big}. {The stimulated absorptivity arises from the blocking factor $(1-f)$ that accounts for Pauli's principle in the difference of emissions and absorptions
$$
(1-f)j-\chi f=j-\tilde\chi f
$$
introduced in \cite{UehlingUhlenbeck33}, see also \cite[pp.\;39/40]{CercignaniKremer02}.} Formulas for the nonnegative expressions $j(\omega,\vec u)$ and $\chi(\omega,\vec u)$, that are nonlinear in $\omega$ and in $\vec u$, for both electron neutrino and its antiparticle have been derived in \cite[pp.\;822--826]{Bruenn85}.

By the term $\mathcal C(\vec u,f)$, the right hand side of \eqref{Boltz} also accounts for isoenergetic scattering of neutrinos (or antineutrinos) on protons, neutrons and nuclei. It is a linear integral operator in $f$, the so-called collision integral, and reads
\begin{equation}
\label{isoscattering}
%
%
\mathcal C\big(\vec u,f(t,r,\mu,\omega)\big)
=\frac{\omega^2}{c(hc)^3}\left[\int_{-1}^{1} R\big(\vec u,\mu,\mu',\omega\big) \big(f(t,r,\mu',\omega)-f(t,r,\mu,\omega)\big) d\mu'\right]
\end{equation}
where the isoenergetic scattering kernel $R(\vec u(t,r),\mu,\mu',\omega)$ is symmetric in $\mu$ and $\mu'$ and depends nonlinearly on all its entries, see \cite[pp.\;806/7,\;826--828]{Bruenn85} for concrete formulas. {As in the stimulated absorptivity, the quantum aspect of neutrinos is taken into account by blocking factors that are hidden in \eqref{isoscattering}, see  \cite[p.\;1176]{LiebendoerferEtAl09big}.
} Furthermore, Planck's constant is denoted by~$h$, the term corresponding to $f(t,r,\mu',\omega)$
accounts for in-scattering while the term corresponding to $f(t,r,\mu,\omega)$ represents out-scattering. Note that if $f$ does not depend on $\mu$ we have $\mathcal{C}(f)=0$. As an immediate consequence of the symmetry of the collision kernel 
with respect to $\mu$ and $\mu'$ we obtain for any $f$
\begin{equation}
\label{collision_vanishes}
\int_{-1}^1\mathcal{C}(\vec u,f)\,d\mu=0\,.
\end{equation}

{Neutrino-neutrino interactions and further neutrino interactions with the background matter as considered in~\cite[p.\;774]{Bruenn85} such as, e.g., neutrino-electron scattering are neglected in \cite{LiebendoerferEtAl09big}. For a more comprehensive interaction list see \cite{LentzEtAl12}. The neutrino interactions with matter that contribute to the opacity will be defined in Subsection~\ref{Legendre} for our case.}


\subsection{Isotropic Diffusion Source Approximation (IDSA)}
\label{IDSA}

In this section we give a short introduction to the Isotropic Diffusion Source Approximation (IDSA) that has been developed in \cite{LiebendoerferEtAl09big}. The aim of this approximation of Boltzmann's equation~\eqref{Boltz} is to reduce the computational cost for its solution, making use of the fact that \eqref{Boltz} 
is mainly governed by diffusion of neutrinos in the inner core and by 
transport of free streaming neutrinos in the outer layers of a star. The following ansatz for the IDSA intends to avoid the solution of the full Boltzmann equation in a third domain in between these two regimes as well as the detection of the corresponding domain boundaries. 
{In what is to come, we drop the dependency on $\vec u$ in the notation of the Boltzmann equation and abbreviate the right hand side of \eqref{Boltz} by $j+\mathcal J(f)$ where $\mathcal J(f)$ is linear in~$f$.}

\subsubsection{Ansatz: Decomposition into trapped and streaming neutrinos}
\label{ansatz}

We assume a \emph{decomposition} of
\begin{equation}
\label{decomposition}
f= f^t + f^s
\end{equation}
\emph{on the whole domain} into distribution functions $f^t$ and $f^s$ supposed to account for \emph{trapped} and for \emph{streaming} neutrinos, respectively.

With this assumption and using the linearity \mbox{$\mathcal D(f) = \mathcal D(f^t) + \mathcal D(f^s)$} as well as $\mathcal J(f) = \mathcal J(f^t)+ \mathcal J(f^s)$, solving the Boltzmann equation \eqref{Boltz}, i.e., 
$$
\mathcal D(f)=j+\mathcal J(f)
$$
is equivalent to solving the two equations
\begin{align}
\mathcal D(f^t) &= j+\mathcal J (f^t) -\Sigma\label{decomposedtrapped}\\[2mm]
\mathcal D(f^s) &= \mathcal J (f^s) +\Sigma\label{decomposedstreaming}
\end{align}
with an arbitrary coupling term $\Sigma(r,t,\mu, \omega,f^t,f^s,\vec u)$. For the IDSA one establishes approximations of these two equations arising from physical properties of trapped and streaming particles, respectively, and one determines an appropriate coupling function $\Sigma(r,t,\mu,\omega,f^t,f^s,\vec u)$. Following Remark~\ref{supernova}, the aim is to deal only with angular moments of $f^t$ and $f^s$ instead of the full distribution functions themselves.

\subsubsection{Hypotheses and their consequences for the trapped and streaming particle equations}
\label{hypotheses}

Concretely, one uses the following hypotheses. First, the trapped particle component $f^t$ as well as $\Sigma$ are assumed to be \emph{isotropic}, i.e., independent of $\mu$. 
Taking the angular mean of equation \eqref{decomposedtrapped} this leads to the trapped particle equation
\begin{equation}
\label{eq:trapped}
\frac{d f^t}{cd t} + \frac{1}{3}\frac{d \ln\rho}{cd t}\omega\frac{\partial f^t}{\partial \omega} = j-\tilde \chi f^t -\Sigma\,, 
\end{equation}
in which we slightly abuse the notation 
\begin{equation}
\label{f_t_is_beta_t}
f^t=\frac{1}{2}\int_{-1}^1 f^t d\mu
\end{equation}
concerning the domain of definition of the isotropic $f^t$ given in~\eqref{fDef}. The isotropic source term $\Sigma$ is treated in the same way. Here, and in what follows, we always assume that we can interchange the integral and the differentiation operators. Details of the computation can be found in the proof of Lemma~\ref{Liebend_approx}.

Next, $f^t$ is assumed to be in the \emph{diffusion limit}, which is physically at least justified for the inner core of the star. 
In order to derive the diffusion limit, a Legendre expansion of the scattering kernel $R(\mu,\mu',\omega)$ with respect to its angular dependence, truncated after the second term, is used in \cite[App.\;A]{LiebendoerferEtAl09big} for an approximation of the collision integral, see Subsection~\ref{Legendre}. In fact, this approximation is essential for the derivation of the diffusion limit and thus the resulting definition of the source term $\Sigma$. 
In \cite{BeFrGaLiMiVa12_esaimbig}, this derivation based on a Chapman--Enskog-like expansion of \eqref{Boltz} is performed in detail. In case of very high densities, even the diffusion term can be neglected. Then $f^t$ is in the reaction limit and satisfies a reaction equation. It is the purpose of this paper to give derivations of the diffusion, the reaction and the free streaming limit, to which we turn in the following, by mathematically concise applications of Chapman--Enskog expansions and Hilbert expansions, see Section~\ref{asymptotics}.



Secondly, the streaming particle component $f^s$ is assumed to be in the \emph{free streaming limit}. This justifies to neglect the collision integral in \eqref{decomposedstreaming}, which by \eqref{collision_vanishes}, however, vanishes anyway after angular integration. Furthermore, it justifies to neglect the dynamics of background matter so that one can use the laboratory frame formulation \eqref{Boltz_frozen} of \eqref{Boltz} with frozen matter as the left hand side of~\eqref{decomposedstreaming} (here, we also neglect the Lorentz transformation). For the same reason one can assume the free streaming particle component to be in the \emph{stationary state limit} and drop the time derivative in~\eqref{Boltz_frozen} which then, after angular integration, becomes the streaming particle equation
\begin{equation}
\label{eq:streaming}
\frac{1}{r^2}\frac{\partial}{\partial r}\left( \frac{r^2}{2}\int_{-1}^1f^s\mu d\mu\right) = -\frac{\tilde\chi}{2}\int_{-1}^1f^s d\mu+\Sigma\,.
\end{equation}
Due to \eqref{nabla_div_in_r} this equation can be reformulated as a Poisson equation 
for a spatial potential $\psi$ of the first angular moment of $f^s$.

For the solution of \eqref{eq:streaming}, the approximate relationship that has been suggested by Bruenn in~\cite{LiebendoerferEtAl04big},
\begin{equation}
\label{scattering_sphere}
\frac{1}{2}\int_{-1}^1f^s(\omega) \mu d\mu
=\frac{1}{2}\left(1+\sqrt{1-\left(\frac{R_\nu(\omega)}{\max(r,R_\nu(\omega))}\right)^2}\right)\frac{1}{2}\int_{-1}^1f^s(\omega) d\mu\, ,
\end{equation}
between the particle flux and the particle density of streaming neutrinos is used. Here, \mbox{$R_\nu(\omega)>0$} is the energy dependent radius of the neutrino scattering spheres. In addition to spherical symmetry, this approximation is based on the assumption that all free streaming particles of a given energy $\omega$ are emitted isotropically at their corresponding scattering sphere~\cite[p.\;1178]{LiebendoerferEtAl09big}. Hence, the flux $\frac{1}{2}\int_{-1}^1f^s(\omega) \mu d\mu$ can be approximated in terms of the particle density $\frac{1}{2}\int_{-1}^1f^s(\omega) d\mu$ by the geometrical relationship~\eqref{scattering_sphere}. The geometrical factor that correlates these terms is the only place in the IDSA where the angular distribution of neutrinos, namely of the free streaming neutrinos, is (approximatively) encoded. This factor is equal to $1/2$ if $r \leq R_{\nu}(\omega)$, which is a consequence of the isotropy of $f$ inside the scattering spheres, and increases up to $1$ in the limit $r \rightarrow \infty$. The latter expresses the fact that the neutrinos tend to stream radially outwards so that the distribution function $f$ accumulates at $\mu= 1$.

\subsubsection{Legendre expansion of the scattering kernel}
\label{Legendre}

As mentioned above, we now seek for an approximation of the collision integral suited to derive the diffusion limit and also the reaction limit of Boltzmann's equation that we envisage in the next section. The approximation is performed by a Legendre expansion of the scattering kernel. For an introduction to Legendre expansions by spherical harmonics we refer to \cite[pp.\;302, 391--395]{WhittakerWatson80}. Concretely, the Legendre series for $\frac{\omega^2}{c(hc)^3}R(\mu,\mu',\omega)$ reads
\begin{equation}
\label{Legendre_series}
\frac{\omega^2}{c(hc)^3}R(\mu,\mu',\omega) = \frac{1}{4\pi} \sum_{l=0}^{\infty}(2l+1)\phi_l(\omega)\int_{0}^{2\pi}P_l(\cos\theta)d\varphi\, ,
\end{equation}
with the Legendre polynomials $P_l$, $l=0,1,\ldots$, where $\theta$ is the angle between the incoming and the outgoing particle and 
\begin{equation}
\label{cos_in_Legendre}
\cos(\theta) = \mu\mu' + [(1-\mu^2)(1-\mu'^2)]^{\frac{1}{2}}\cos\varphi\,.
\end{equation}
With the first two Legendre polynomials given by
\begin{equation*}
P_0(\cos \theta) = 1\,,\quad
P_1(\cos \theta) = \cos\theta\,, 
\end{equation*}
truncation of the series after the second term gives the approximation
\begin{equation*}
\frac{\omega^2}{c(hc)^3}R(\mu,\mu',\omega) \approx \frac{1}{4\pi}\left(\phi_0(\omega)\int_0^{2\pi}1\, d\varphi + 3 \phi_1(\omega)\int_0^{2\pi}\cos(\theta)d\varphi\right)
=\frac{1}{2}\phi_0(\omega) + \frac{3}{2}\phi_1(\omega)\mu\mu'\,.
\end{equation*}
Inserting this into the collision integral \eqref{isoscattering}, without explicitly mentioning the dependency on $t$, $r$ and $\omega$, one obtains
\begin{equation}
\mathcal C(f) \approx \frac{1}{2}\int_{-1}^1 (\phi_0 + 3\phi_1\mu\mu')(f(\mu')-f(\mu))d\mu' = -\phi_0 f+\phi_0\frac{1}{2}\int_{-1}^1 fd\mu + 3\mu\phi_1\frac{1}{2}\int_{-1}^1 f\mu d\mu\, ,  \label{coll}
\end{equation}
which is an affine function in $\mu$ expressed in terms of $f$, the zeroth and first angular moments of $f$ and the coefficients $\phi_0$, $\phi_1$. {Together with the stimulated absorptivity~$\tilde\chi$, the latter give rise to the definition of the total neutrino \emph{opacity} $\tilde\chi+\phi_0-\phi_1$ and the definition of the neutrino \emph{mean free path}
\begin{equation}
\label{meanfreepath}
\lambda:= \frac{1}{j+\chi+\phi_0-\phi_1}=\frac{1}{\tilde\chi+\phi_0-\phi_1}\,,
\end{equation}
consult \cite[p.\;640]{MezzacappaBruenn93abig} and \cite[pp.\;1177,\,1188]{LiebendoerferEtAl09big}. 
This definition of the mean free path} is motivated by the fact that $\lambda/3$ occurs as the diffusion parameter in the diffusion limit of the Boltzmann equation that will be derived in the following section by asymptotic expansions (see \eqref{diff_react_weak}). It is clear that the smaller the diffusion parameter is, the smaller the diffusion of neutrinos represented by the diffusion term in \eqref{diff_react_weak} becomes which physically corresponds to a smaller mean free path. 

Finally, we remark that a collision kernel which can be expanded as in \eqref{Legendre_series} with \eqref{cos_in_Legendre} is always symmetric in $\mu$ and $\mu'$ since $\cos(\theta)$ in \eqref{cos_in_Legendre} has this property. For the same reason, \eqref{coll} as well as any truncation of the Legendre series in \eqref{Legendre_series} is symmetric in $\mu$ and $\mu'$. In particular, property \eqref{collision_vanishes} also holds for the approximation of the collision integral in~\eqref{coll}.

\subsubsection{Diffusion source and system of coupled equations}
\label{diffusion_source_coupled}

Using the truncated Legendre expansion of the collision kernel as in the previous subsection, the \emph{isotropic diffusion source} 
$\Sigma=\Sigma_{\rm{ids}}$ 
is derived in \cite[App.\;A]{LiebendoerferEtAl09big} by a Chapman--Enskog--like asymptotic expansion. The derivation that provides the diffusion source in a certain ``leading order approximation'' is explained in detail in~\cite{BeFrGaLiMiVa12_esaimbig}. Here, we quote the result which we will reobtain by Chapman--Enskog and Hilbert expansions in Section~\ref{asymptotics} where we will also explain the precise meaning of the ``leading order approximation''. The diffusion source reads
\begin{equation}
\label{Sigma_diff}
\Sigma_{\rm{ids}}:=-\frac{1}{r^2}\frac{\partial}{\partial r}\left(r^2\frac{\lambda}{3}\frac{\partial f^t}{\partial r}\right) + \frac{\tilde\chi}{2}\int_{-1}^1f^s d\mu\,.
\end{equation}
We emphasize that, in spite of \eqref{collision_vanishes}, the influence of the collision integral is contained ``in leading order'' in $\Sigma\left(\omega,f^t,\frac{1}{2}\int_{-1}^1 f^s d\mu,\vec u\right)$ by its dependence on $\vec u$ which contains the mean free path $\lambda$ and also the emissivity $j$ and the {stimulated absorptivity}~$\tilde\chi$.

By physical reasons outlined in \cite[p.\;1177]{LiebendoerferEtAl09big} and also justified in Subsections \ref{Sigma_free_streaming} and~\ref{Sigma_reaction}, the coupling term $\Sigma$ is limited from above by $j$ and from below by $0$, i.e., for the coupling of the trapped particle equation \eqref{eq:trapped} and the streaming particle equation \eqref{eq:streaming} the source term
\begin{equation}
\label{Sigma}
\Sigma := \min \left\{\max\left[\Sigma_{\rm{ids}}\,,\, 0\right],\, j \right\}
\end{equation}
is used in \cite{LiebendoerferEtAl09big}. As a result, one obtains the \emph{coupled system of equations} \eqref{eq:trapped}, \eqref{eq:streaming}, \eqref{Sigma} as the \emph{Isotropic Diffusion Source Approximation IDSA} of Boltzmann's equation~\eqref{Boltz}.

The \emph{limiters} $\Sigma=j$ in \eqref{eq:streaming} and $\Sigma=0$ in \eqref{eq:trapped} can be regarded as representing the free streaming limit and the reaction limit, respectively, of the Boltzmann equation~\eqref{Boltz}. In Section~\ref{asymptotics}, we will also derive these limits by Chapman--Enskog and Hilbert expansions under certain assumptions. A~mathematical motivation as to why these limits enter formula \eqref{Sigma} in the form of limiters for the source term $\Sigma$ will be addressed at the end of this paper (see Lemma~\ref{limiter_lemma}).

As a consequence of the limiters $j$ and $0$ in \eqref{Sigma}, one can consider three different regimes of the system \eqref{eq:trapped}, \eqref{eq:streaming} and \eqref{Sigma} depending on the three alternatives $\Sigma$ can assume in formula~\eqref{Sigma}. We call these regimes according to the corresponding limits of the Boltzmann equation \eqref{Boltz} and emphasize that transient situations that lie in between the limits are treated in the IDSA as belonging to one of the three regimes. In this sense the terms ``regime'' and ``limit'' have different meanings in the following.

\subsubsection{Diffusion regime: $\Sigma = \Sigma_{\rm{ids}}$}
\label{diffusion_source}

{If we consider the formulation in spherical symmetry in~\eqref{nabla_div_in_r}, we see that the first summand in~\eqref{Sigma_diff} represents a diffusion term. Moving this term to the left hand side of~\eqref{eq:trapped} gives a diffusion-reaction-type equation
\begin{equation}
\label{eq:trapped_diff}
\frac{d f^t}{cd t} + \frac{1}{3}\frac{d \ln\rho}{cd t}\omega\frac{\partial f^t}{\partial \omega}-\frac{1}{r^2}\frac{\partial }{\partial r} \left(r^2\frac{\lambda}{3}\frac{\partial f^t}{\partial r}\right) = j-\tilde \chi\left( f^t+\frac{1}{2}\int_{-1}^1 f^s d\mu\right)
\end{equation}
for $f^t$, balanced by the emissivity $j$ and an effective absorption $-\tilde \chi\left(f^t+\frac{1}{2}\int_{-1}^1 f^s d\mu\right)$ of trapped and streaming neutrinos.} Consequently, in the diffusion regime there is only absorption but no emission of streaming particles. However, trapped particles can be converted into streaming particles. The equation for the latter reduces to
\begin{equation}
\label{eq:streaming_diff}
\frac{1}{2}\int_{-1}^1f^s\mu d\mu=-\frac{\lambda }{3}\frac{\partial f^t}{\partial r}\,,
\end{equation}
i.e., $-\frac{\lambda}{3} f^t$ is a potential of the streaming particle density. In the diffusion limit, where one can assume $r<R_\nu$ in \eqref{scattering_sphere}, this flux is just half of the particle density. 
The diffusion regime turns out to be the only regime where the equations for $f^t$ and $f^s$, i.e., \eqref{eq:trapped_diff} and \eqref{eq:streaming_diff}, are effectively coupled.

Since the basic ansatz \eqref{decomposition}--\eqref{decomposedstreaming} for the IDSA assumes an additive coupling term $\Sigma$ in the equations for the trapped and the streaming particle component on the whole domain, i.e., also in regimes where \eqref{Boltz} cannot be supposed to be in the diffusion limit, the diffusion regime of the IDSA is likely to be used where the diffusion limit does not apply. For more extreme cases, i.e., where the limiters in \eqref{Sigma} apply, the diffusion source is altered in order to account for other limits of the Boltzmann equation to which we turn now.

\subsubsection{Free streaming regime: $\Sigma=j$}
\label{Sigma_free_streaming}

If we neglect the dynamics of the background matter, i.e., if we assume $\rho$ to be constant in~\eqref{eq:trapped},  
this equation has a solution $f^t$ that decreases exponentially if we have $\Sigma=j$ and even more so if $\Sigma>j$. Therefore, the trapped particle component $f^t$ vanishes quickly in the case $\Sigma\geq j$, so that its equation \eqref{eq:trapped} as part of the Boltzmann equation \eqref{Boltz} can eventually be neglected. However, neglecting the right hand side of \eqref{eq:trapped} with vanishing $f^t$ necessarily leads to $\Sigma=j$ so that \eqref{Boltz} is eventually approximated by equation \eqref{eq:streaming} for the streaming particle component, in which one needs to have $\Sigma=j$. Physically, one can argue that for large mean free paths~$\lambda$, the streaming particle component is supposed to dominate the trapped particle component and all emitted neutrinos should directly become streaming. The case $\Sigma>j$ would result in even faster decay of $f^t$ and an effective additional unphysical source of streaming particles in \eqref{eq:streaming}. Therefore, the coupling term $\Sigma$ in \eqref{Sigma} is limited from above by $j$.

Consequently, in the free streaming regime $\Sigma=j$ the equation for the trapped particles reads
\begin{equation*}
\frac{d f^t}{cd t} + \frac{1}{3}\frac{d \ln\rho}{cd t}\omega\frac{\partial f^t}{\partial \omega} = -\tilde \chi f^t 
\end{equation*}
while the equation for the streaming particles is given by 
\begin{equation*}
\frac{1}{r^2}\frac{\partial}{\partial r}\left(\frac{r^2}{2}\int_{-1}^1f^s\mu d\mu\right) = j-\frac{\tilde\chi}{2}\int_{-1}^1 f^s d\mu\,.
\end{equation*}

The two equations are obviously uncoupled. Finally, we note that one would expect the free streaming limit to occur outside the neutrinosphere, i.e., with $r>R_\nu$ in \eqref{scattering_sphere}.

\subsubsection{Reaction regime: $\Sigma = 0$}
\label{Sigma_reaction}

One can argue that the lower limiter $\Sigma = 0$ in \eqref{Sigma} represents the regime in which the roles of $f^t$ and $f^s$ are changed compared to the streaming regime, reflected by the right hand sides of the corresponding equations. Now the trapped particles dominate the streaming particles which can only be absorbed. 

The equation for the trapped neutrinos reads
\begin{equation}
\label{trapped_reaction}
\frac{d f^t}{cd t} + \frac{1}{3}\frac{d \ln\rho}{cd t}\omega\frac{\partial f^t}{\partial \omega} = j-\tilde \chi f^t 
\end{equation}
while the equation for the streaming neutrinos takes the form
\begin{equation*}
\frac{1}{r^2}\frac{\partial}{\partial r}\left(\frac{r^2}{2}\int_{-1}^1f^s\mu d\mu\right) = -\frac{\tilde\chi}{2}\int_{-1}^1 f^s d\mu\,.
\end{equation*}
As in the free streaming regime, these equations are uncoupled. The case $\Sigma<0$ would result in an additional source for trapped neutrinos induced by the scattering of streaming neutrinos that would eventually lead to an unrealistically big trapped particle component. Note that with $\Sigma=0$, the stationary state of the trapped neutrino equation 
is given by $f^t=j/\tilde\chi$ which is a Fermi--Dirac function, i.e., the right distribution function for the trapped particles in thermal equilibrium, cf.\ \cite[p.\;822]{Bruenn85} and \cite[p.\;1177]{LiebendoerferEtAl09big}. {We call the case $\Sigma=0$ reaction regime since here the reactions \eqref{reactions} encoded in $j$ and $\tilde\chi$ drive the time evolution of $f^t$ together with the
second term on the left hand side of \eqref{trapped_reaction}. The latter describes the redistribution of neutrinos in energy space if the gas is
compressed or expanded (and it occurs in the trapped particle equation for the diffusion and the free streaming regime, too).} 
Finally, as in the diffusion limit, we expect the reaction limit to prevail inside the neutrinosphere, i.e., with $r<R_\nu$ in \eqref{scattering_sphere}. 

{
\begin{remark}
Note that the right hand side of \eqref{trapped_reaction} has the same structure as the right hand side of the model equations~\cite[Sec.\;8]{CercignaniKremer02}. In fact, if we neglect collisions in the Boltzmann equation~\eqref{Boltz}, the stimulated absorptivity $\tilde\chi=j+\chi$ turns out to be an approximation of the neutrino opacity, i.e., the inverse of the mean free path $\lambda$ in~\eqref{meanfreepath}. The right hand side of \eqref{trapped_reaction} can then be approximated by $(j/{\tilde\chi}-f^t)/\lambda$ with the equilibrium function~$j/{\tilde\chi}$. Since the Boltzmann equation \eqref{Boltz} as well as \eqref{trapped_reaction} are scaled with the factor $c^{-1}$, we can write $\lambda=c\tau$ with a mean free time $\tau$ and identify $\tau$ as the relaxation time if we interpret the term $(j/{\tilde\chi}-f^t)/\lambda$ as the right hand side of a model equation.
\end{remark}
}


\subsubsection{Mathematical issues}
\label{issues}

The main problem of the IDSA is its treatment of transient regions in between the parts of the domain where neither the diffusion, the free streaming nor the reaction limit applies. Then, nevertheless, one of the above regimes occurs. A discussion of this topic by numerical results comparing the IDSA with the full Boltzmann equation in the spherically symmetric case is given in~\cite{LiebendoerferEtAl09big}. In particular, the transition from the diffusive to the free streaming limit in a ``semi-transparent'' regime across the neutrino sphere is quite sensitive and physically essential. The numerical results are satisfactory but exhibit the biggest deviations between the IDSA and the full Boltzmann equation in these transition regimes. Similar results have been found in a numerical example for a model problem in~\cite{BeFrGaLiMiVa12_esaimbig}.

In \cite{LiebendoerferEtAl09big} the diffusion limit is derived by a ``leading order'' Chapman--Enskog--like expansion of the Boltzmann equation \eqref{Boltz} in which some higher order terms are neglected. However, the solvability hypothesis of the Chapman--Enskog expansion~\eqref{beta_0=beta} that requires the zeroth angular moment of $f$ to be equal to the zeroth angular moment of $f_0$ might be violated by physical reasons since $f_0$ only depends on stationary quantities whereas $f$ represents the conservative evolution of neutrinos in time. Further details concerning these issues, e.g., a weakened form of the solvability hypothesis, are given in Subsection~\ref{Chapman--Enskog}. In Subsection~\ref{Hilbert}, we present a derivation of the diffusion limit by a Hilbert expansion of the time-scaled Boltzmann equation~\eqref{Boltz} without the need to neglect terms.

The limiters for the coupling term in \eqref{Sigma} can be physically and mathematically motivated (see above and Lemma~\ref{limiter_lemma}), but there does not seem to be a rigorous explanation of these limitations by means of the full Boltzmann equation. In Subsection~\ref{Hilbert}, we obtain these limiters, which represent the free streaming and the reaction limits, respectively, by corresponding Hilbert expansions of \eqref{Boltz}.

Besides the fact that the density range in a core-collapse supernova comprises several orders of magnitude, essential processes described by the coupled problem \eqref{Hydro}--\eqref{Boltz} occur on different time scales that, moreover, depend on where in the star the processes take place. Concretely, the reaction and collision time scale given by the interaction rates $j$ and $\tilde\chi$ as well as the opacities $\phi_0$ and $\phi_1$ describes the fastest processes of \eqref{Boltz} at the center of the star. In the semi-transparent regime it becomes slower than the transport time scale represented by the second term in \eqref{Boltz} as well as the term $F_\mu^0$. In the outer layers it also becomes slower than the hydrodynamic time scale given by the terms involving $\rho$ and $v$ in \eqref{Boltz}. The hydrodynamic time scale is slowest at low density and reaches almost $c$ at the center. In contrast, the diffusion time scale is slowest at the center and reaches almost $c$ in the streaming regime. We also use time scaling in the Hilbert expansions for the derivations of the reaction, diffusion and free streaming limits in Subsection~\ref{Hilbert}.


\section{Asymptotic Behaviour of Radiative Transfer and IDSA}
\label{asymptotics}

The main purpose of this section is to present different derivations of the diffusion source $\Sigma=\Sigma_{\rm{ids}}$ given in Subsection \ref{diffusion_source} with the help of asymptotic analysis. In a first approach, we apply a Chapman--Enskog expansion to Boltzmann's equation~\eqref{Boltz}. Using weakened Chapman--Enskog hypotheses, we give an explanation of how to obtain the results for the diffusion limit in \cite[App.\;A]{LiebendoerferEtAl09big} within this framework. As a by-product, we also see how the reaction limit in Subsection~\ref{Sigma_reaction} can be explained in terms of a Chapman--Enskog expansion. 

In a second approach, we apply Hilbert expansions to~\eqref{Boltz} which for the reaction and the diffusion limit exhibit strong correlations with the Chapman--Enskog expansions while at the same time they require less assumptions. In this framework, we also introduce variations of different scalings for time and for the right hand side of the Boltzmann equation. This allows us to quite naturally rederive the diffusion limit, but now up to a ``clean'' $\BigO(\eps^2)$ order of the scaled Boltzmann equation with the scaling parameter $\eps>0$. Furthermore, using a Hilbert expansion just for the time scaled Boltzmann equation, we derive the free streaming limit as well as the limiter $\Sigma=j$ for the streaming regime given in Subsection~\ref{Sigma_free_streaming}.

\subsection{Chapman--Enskog expansion}
\label{Chapman--Enskog}

The Chapman--Enskog expansion is a well-established method for solving the Boltzmann equation for non-uniform gases and in order to derive the Euler and the Navier--Stokes equations as its first and second order approximations. It was found independently by S.~Chapman \cite{Chapman16big,Chapman17big} and D.~Enskog~\cite{Enskog17} in the 1910's. We apply it to our version \eqref{Boltz} of the Boltzmann equation for neutrinos. Our use of the method is based on the presentation in~\cite[pp.\;115--122]{FerzigerKaper72} that focusses on Enskog's approach.

If the gas of neutrinos is in equilibrium, the distribution function is given by a Fermi--Dirac distribution, see, e.g., {\cite[p.\;52]{CercignaniKremer02} or}~\cite[p.\;1179]{LiebendoerferEtAl09big}. Concretely, this means that $f=j/\tilde\chi$ and $\mathcal{D}(f)=\mathcal{C}(f)=0$ in \eqref{Boltz}, compare \cite[p.\;1177]{LiebendoerferEtAl09big}. In case of small deviations from the equilibrium, we expect the emission and absorption rates as well as the opacities in the collision integral $\mathcal{C}(f)$, i.e., the right hand side of \eqref{Boltz}, to remain the dominant terms in \eqref{Boltz} that drive the gas to equilibrium. (This reflects the fact that the corresponding processes occur on the fastest time scale.) Therefore, a small parameter $\varepsilon>0$ is introduced that emphasizes the weight of these terms on the right hand side of the Boltzmann equation, i.e., we write
\begin{equation}
\label{Boltzmann_eps}
\mathcal D(f) = \varepsilon^{-1} \left(\bar{j}+ \bar{\mathcal J}(f) \right)\,.
\end{equation}
Here, we denote $\bar{j}= \varepsilon j$ as well as $\bar{\tilde\chi}=\varepsilon\tilde\chi$ and $\bar{R}=\varepsilon R$, or, in particular, $\bar{\phi}_0=\varepsilon\phi_0$ and $\bar{\phi}_1=\varepsilon\phi_1$ in the case~\eqref{coll}, for the definition of $\bar{\mathcal J}=\varepsilon\mathcal J$ according to \eqref{Boltz} and~\eqref{isoscattering}. We assume the terms with a bar to be independent of~$\varepsilon$ which for small $\varepsilon$ results in a small mean free path $\lambda=\eps (\bar{\tilde{\chi}}-\bar\phi_0+\bar\phi_1)^{-1}$ according to~\eqref{meanfreepath}. 
{
\begin{remark}
Equation \eqref{Boltzmann_eps} is an example of a scaled Boltzmann equation in which the small parameter $\varepsilon$ is known as the Knudsen number, cf.~\cite{SpeckStrain10} or~\cite[p.\;128]{CercignaniKremer02}. The Knudsen number can be regarded as the ratio of the mean free path and the considered macroscopic length. In the diffusion and reaction regimes this value will be small, whereas in the free streaming regime it can be of order $\BigO(1)$. Therefore, in the scalings introduced in the next section for the free streaming regime, the parameter $\varepsilon$ can no longer be interpreted as the Knudsen number.
\end{remark}}
 
Now we look for a solution of \eqref{Boltzmann_eps} in the form of a power series expansion with respect to $\varepsilon$, i.e.,
\begin{equation}
\label{f_expanded}
f = f_0 + \varepsilon f_1 + \varepsilon^2 f_2 + \ldots\,.
\end{equation}
Enskog postulated that the distribution function and, therefore, all $f_i$ functions do not depend explicitly on time but only implicitly via their dependence on macroscopic observables (and their spatial gradients) obtained by velocity moments of $f$ weighted with the collisional invariants. This can be justified both mathematically and physically \cite[p.\;116]{FerzigerKaper72}. Here, we do not analyze the analogue of collisional invariants for our case \eqref{Boltz} but restrict ourselves to the trivial invariant function $1$ (which is invariant with respect to any variable). Therefore, we only consider the velocity moment 
$$
\beta := \frac{1}{2}\int_{-1}^1 f d\mu\quad\mbox{with}\quad \beta=\beta(r,\omega, t)
$$ 
and impose the assumption 
\begin{equation}
\label{f_dep_beta}
f(r, \mu, \omega, t) = f\big(r, \mu, \omega \,|\, \beta, \mbox{$\frac{\partial\beta}{\partial r}$}, \mbox{$\frac{\partial^2\beta}{\partial r^2}$},\ldots\big)
\end{equation}
as our equivalent of Enskog's postulate. We can regard $\beta(r,t,\omega)$ as a measure for the particle density of neutrinos for a certain energy $\omega$. Its definition and the expansion of $f$ provides $\beta_i$, $i\in\N_0$, with $\beta_i=\frac{1}{2}\int_{-1}^1 f_i d\mu$ such that we formally obtain the series 
\begin{equation}
\label{beta_expanded}
\beta = \beta_0 + \varepsilon \beta_1 + \varepsilon^2 \beta_2 + \ldots\,,
\end{equation}
and thus the flux $\Phi$, defined by
$$
\Phi\big(r,\omega \,|\, \beta, \mbox{$\frac{\partial\beta}{\partial r}$}, \mbox{$\frac{\partial^2\beta}{\partial r^2}$}, \ldots\big):=\frac{\partial \beta}{\partial t} =\frac{\partial \beta_0}{\partial t}+\varepsilon\frac{\partial \beta_1}{\partial t}+\varepsilon^2\frac{\partial \beta_2}{\partial t}+ \ldots\,.
$$
Consequently, taking the chain rule into account, we express the partial time derivative of $f$ by the formal series
$$
\frac{\partial f}{\partial t}
=\sum_{k=0}^\infty\frac{\partial f}{\partial \frac{\partial^k\beta}{\partial r^k}}\cdot\frac{\partial}{\partial t}\frac{\partial^k\beta}{\partial r^k}
=\sum_{i=0}^\infty\varepsilon^i\sum_{k=0}^\infty\frac{\partial^k}{\partial r^k}\frac{\partial\beta_i}{\partial t}\cdot\frac{\partial f}{\partial \frac{\partial^k\beta}{\partial r^k}}\,.
$$
Inserting the series \eqref{f_expanded} for $f$ in here leads to the Cauchy product
\begin{equation}
\label{dt_Cauchy}
\frac{\partial f}{\partial t} = \frac{\partial_0 f_0}{\partial t} + \varepsilon \left[ \frac{\partial_1 f_0}{\partial t}+\frac{\partial_0 f_1}{\partial t}\right] + \varepsilon^2 \left[ \frac{\partial_2 f_0}{\partial t}+\frac{\partial_1 f_1}{\partial t}+\frac{\partial_0 f_2}{\partial t}\right] + \dots
\end{equation}
where the operators $\frac{\partial_i}{\partial t}$ are defined by
\begin{equation}
\label{partial_i}
\frac{\partial_i}{\partial t}:=\sum_{k=0}^\infty\frac{\partial^k}{\partial r^k}\frac{\partial\beta_i}{\partial t}\cdot\frac{\partial }{\partial \frac{\partial^k\beta}{\partial r^k}}\,,\quad\quad i=0,1,\dots\,,
\end{equation}
with the coefficients $\frac{\partial\beta_i}{\partial t}$ from the flux expansion.

With the expansion \eqref{dt_Cauchy}, \eqref{partial_i} 
of the partial time derivative operator applied to~$f$, 
we can now expand the operator on the left hand side of \eqref{Boltz} and write 
\begin{equation}
\label{D_Chapman_Enskog}
\mathcal D(f) = \mathcal D(f)^{(0)} +\varepsilon\mathcal D(f)^{(1)} +\varepsilon^2\mathcal D(f)^{(2)} +\dots
\end{equation}
with 
\begin{equation}
\label{Di_Chapman_Enskog}
\mathcal D(f)^{(i)} =  \sum_{l=0}^i\frac{\partial_{i-l}f_{l}}{c\partial t}
+v\frac{\partial f_{i}}{c\partial r}+ \mu \frac{\partial f_{i}}{\partial r} +F_\mu\frac{\partial f_{i}}{\partial \mu}+F_\omega\frac{\partial f_{i}}{\partial \omega}\,,
\end{equation}
using \eqref{Lagrange}. The right hand side of \eqref{Boltzmann_eps} is formally expanded by linearity of $\bar{\mathcal J}$, i.e., 
\begin{equation*}
\bar{\mathcal J}(f) = \bar{\mathcal J}(f_0) +\varepsilon\bar{\mathcal J}(f_1) +\varepsilon^2\bar{\mathcal J}(f_2) +\dots\,.
\end{equation*}

Collecting all terms of the same order $i-1$, $i=0,1,\ldots$, in $\varepsilon$ on both sides of the expanded equation \eqref{Boltzmann_eps} we get the hierarchy of equations
\begin{eqnarray}
\label{CEHie1}
0&=&j+\mathcal J(f_0)\,,\\[1mm]
\label{CEHie2}
\mathcal D(f)^{(i-1)} &=& \bar{\mathcal J}(f_i)\,,\quad\quad i = 1,2, \dots\,.
\end{eqnarray}
These represent integral equations for the successive terms $f_i$ in the power series
expansion of $f$. 
As in the Chapman--Enskog expansion for non-uniform gases, necessary conditions for the solvability of \eqref{CEHie1} or~\eqref{CEHie2} are obtained by multiplication of these equations with the invariant function $1$ and integration over $\mu$. 
As a result, by \eqref{collision_vanishes}, we get the \emph{compatibility equations}
\begin{eqnarray}
\label{CEHie_int_1}
0 &=& j -\frac{\tilde\chi}{2}\int_{-1}^1 f_0 d\mu\,,\\[2mm]
\frac{1}{2}\int_{-1}^{1}\mathcal D(f)^{(i-1)}d\mu &=& -\frac{\bar{\tilde\chi}}{2}\int_{-1}^1 f_i d\mu\,,\quad\quad i = 1,2,\dots\,.
\label{CEHie_int_2}
\end{eqnarray}
 
\begin{remark}
If we do not apply the sophisticated operator expansion \eqref{D_Chapman_Enskog} but use $\mathcal D(f_i)$ instead of $\mathcal D(f)^{(i)}$, motivated by linearity of $\mathcal D$, the analogous construction leads to the same hierarchy of equations as in a Hilbert expansion with scaled reactions and collisions. See Subsection~\ref{Hilbert} as well as Lemma~\ref{strong_lemma} and Remark~\ref{strong_remark} for the analogies and the differences between the two expansions.
\end{remark}

As mentioned earlier, except for the geometrical relationship \eqref{scattering_sphere}, the IDSA does not rely on information how $f$ behaves with respect to~$\mu$. Therefore, it is enough to consider the angularly integrated Boltzmann equation~\eqref{Boltz} only, compare also Remark~\ref{supernova}. In the following, our aim is to compute an approximation of the left hand side of this angularly integrated equation, which is also known as the total interaction rate $s$ given by
\begin{equation}
\label{sequ0}
s:=\frac{1}{2}\int_{-1}^1\mathcal D(f)d\mu = \frac{1}{2}\int_{-1}^{1}\sum_{i=1}^\infty\left(\varepsilon^{(i-1)}\mathcal D(f)^{(i-1)}\right)d\mu 
=\frac{\bar j}{\varepsilon}- \frac{\bar{\tilde\chi}}{2\varepsilon}\int_{-1}^1\sum_{i=0}^\infty \varepsilon^{i} f_i  d\mu = \frac{\bar j}{\varepsilon}- \frac{\bar{\tilde\chi}}{2\varepsilon}\int_{-1}^1  f  d\mu\,.
\end{equation}
In here, and in what follows, we always use formal series and assume their interchangeability with integral operators and partial derivations and interchangeability of the latter operators. Furthermore, we do not use the notation $\mathcal{O}(\varepsilon^k)$ for the $k$-th order, $k\geq 1$, in a rigid mathematical sense but only to express that we consider the corresponding series up to the order $\varepsilon^{k-1}$ while neglecting the rest of the series that starts with the order~$\varepsilon^k$. Since we want to compute $s$ up to second order, we need to handle angular means of $\mathcal D(f)^{(i)}$. This can be difficult because of the sophisticated treatment of the partial time derivative. Therefore, we base our considerations on the following two lemmas.

\begin{lemma}
\label{strong_lemma}
If, apart from Enskog's postulate \eqref{f_dep_beta}, we assume the \emph{strong Chapman--Enskog hypothesis}
\begin{eqnarray}
\label{beta_0=beta}
\frac{1}{2}\int_{-1}^1f_0 d\mu &=& \beta\,,\\[2mm]
\label{beta_i=0}
\frac{1}{2}\int_{-1}^1f_i d\mu &=& 0\,,\quad\quad i=1,2,\dots\,,
\end{eqnarray}
then we have
$$
\frac{1}{2}\int_{-1}^1\mathcal D(f)^{(i)}d\mu =  \frac{1}{2}\int_{-1}^1\mathcal D(f_i)d\mu\,,\quad\quad i=0,1,\dots\,,
$$
i.e., after angular integration, the Chapman--Enskog expansion has the same form as the Hilbert expansion.
\end{lemma}


\begin{proof}
By the form of $\mathcal{D}(f_i)$, cf.~\eqref{Boltz}, and $\mathcal D(f)^{(i)}$ in \eqref{Di_Chapman_Enskog}, we only need to prove
$$
\frac{1}{2}\int_{-1}^1\left(\mathcal D(f)^{(i)}-
\mathcal D(f_i)\right)d\mu=
\frac{1}{2}\int_{-1}^1 \left(\sum_{l=0}^i\frac{\partial_{i-l}f_{l}}{c\partial t}
-\frac{\partial f_i}{c\partial t}\right)d\mu=0\,,
$$
i.e., after elimination of the constant $c$ and formal interchange of integration and differentiation,
\begin{equation}
\label{ChapmanEnskogISHilbert}
\sum_{l=0}^i\frac{\partial_{i-l}\beta_{l}}{\partial t}
=\frac{\partial \beta_i}{\partial t}\,.
\end{equation}



By \eqref{partial_i} we obtain $\frac{\partial_0 \beta}{\partial t} = \frac{\partial \beta_0}{\partial t}$ since $\beta$ and its spatial derivatives are assumed to be independent variables in our version \eqref{f_dep_beta} of the Enskog postulate. Therefore, all summands in \eqref{partial_i} for $k>0$ vanish. Then, hypothesis~\eqref{beta_0=beta} provides the statement for $i=0$.

For $i>0$ and $0\leq l\leq i$ we infer by \eqref{partial_i} that
\begin{equation}
\label{partial_i-l_vanishes}
\frac{\partial_{i-l} \beta_l}{\partial t} = \frac{\partial \beta_{i-l}}{\partial t}\frac{\partial \beta_{l}}{\partial \beta}\,.
\end{equation}
Here, as for $i=0$, all other summands in \eqref{partial_i} for $k>0$ are zero since $\beta$ and, therefore, all $\beta_i$ do not depend on the spatial derivatives of $\beta$ by postulate~\eqref{f_dep_beta}. In~\eqref{partial_i-l_vanishes} at least one of the terms $\beta_{i-l}$ and $\beta_{l}$ vanishes by hypothesis \eqref{beta_i=0} so that the whole expression becomes zero. Therefore, \eqref{ChapmanEnskogISHilbert} is trivially satisfied.
\end{proof}

\begin{remark}
\label{strong_remark}
The question of unique solvability of the integral equations in the hierarchy \eqref{CEHie1} and~\eqref{CEHie2} is the origin of the strong Chapman--Enskog hypothesis \eqref{beta_0=beta},~\eqref{beta_i=0} in gas dynamics. Concretely, in case of non-uniform gases, the compatibility conditions \eqref{CEHie_int_1} and~\eqref{CEHie_int_2} are not sufficient for unique solvability of the corresponding original hierarchy equations \eqref{CEHie1} and~\eqref{CEHie2}. Instead, one is free to require all the macroscopic information $\beta$ to be already contained in~$f_0$, i.e., the strong Chapman--Enskog hypothesis, and thus one obtains unique solvability of \eqref{CEHie1} and~\eqref{CEHie2}. In contrast, for the Hilbert expansion, the corresponding compatibility conditions are indeed necessary and sufficient conditions for the unique solvability of the corresponding hierarchy equations. Existence of solutions is provided within the theory of Fredholm integral equations by the compatibility conditions that can be regarded as orthogonality conditions for the collisional invariants. Hilbert then proved uniqueness of the solutions if initial values are assigned to all $\beta_i$, $i\in\N_0$. More details can be found in \cite[pp.\;111--113]{FerzigerKaper72} for the Hilbert theory and in \cite[pp.\;118/9]{FerzigerKaper72} for the Chapman--Enskog theory.\\ 
Lemma~\ref{strong_lemma} also holds if $\beta$ is a vector corresponding to collisional invariants since then \eqref{partial_i-l_vanishes}, i.e., the term for $k=0$ in the series \eqref{partial_i}, has the form
\begin{equation*}
\frac{\partial_{i-l} \beta_l}{\partial t} = \frac{\partial \beta_{i-l}}{\partial t}\nabla_\beta\beta_{l}
\end{equation*}
and thus vanishes componentwise, too. We refer to \cite[p.\;117]{FerzigerKaper72} for more details.\\
Finally, even if the form of the Chapman--Enskog and the Hilbert expansions are the same, Enskog's hypothesis~\eqref{f_dep_beta}, which is not used in Hilbert expansions, makes a big difference inasmuch each $f_i$ in the expansion \eqref{f_expanded} of $f$ depends implicitly on every other $f_i$ by its dependence on $\beta$, i.e., on all $\beta_i$, $i=0,1,\ldots$, in~\eqref{f_dep_beta}.
\end{remark}

\begin{lemma}
\label{weak_lemma}
We consider the decomposition
$f = f^t + f^s$ inducing $\beta = \beta^t +\beta^s$ with the angular means $\beta^t=\frac{1}{2}\int_{-1}^1 f^t d\mu$ and $\beta^s=\frac{1}{2}\int_{-1}^1 f^s d\mu$. We assume Enskog's postulate~\eqref{f_dep_beta} and the \emph{weakened Chapman--Enskog hypothesis}
\begin{eqnarray*}
\frac{1}{2}\int_{-1}^1f_0 d\mu &=& \beta^t\,,\quad \mbox{ i.e., }\quad \beta_0=\beta^t\,,\\[2mm] 
\sum_{i=1}^{\infty}\left(\frac{1}{2}\int_{-1}^1\varepsilon^i f_i d\mu\right) &=&  \beta^s\,,\quad \mbox{ i.e., }\quad \sum_{i=1}^{\infty}\varepsilon^i\beta_i=\beta^s\,.
\end{eqnarray*}
Then with the assumptions
\begin{equation}
\label{eps2ass}
\frac{\partial \beta^s}{\partial \beta} = \BigO(\varepsilon^2)\quad\mbox{and}\quad\frac{\partial \beta_1}{\partial \beta}=\BigO(\varepsilon^2)
\end{equation}
we obtain the relations
\begin{equation}
\label{D01CEH}
\frac{1}{2}\int_{-1}^1\mathcal D(f)^{(i)}d\mu =  \frac{1}{2}\int_{-1}^1\mathcal D(f_i)d\mu+\BigO(\varepsilon^2)\,,\quad\quad i\in\{0,1\}\,,
\end{equation}
and with the assumption
\begin{equation}
\label{eps1ass}
\frac{\partial \beta^s}{\partial \beta} = \BigO(\varepsilon)
\end{equation}
we obtain
\begin{equation}
\label{D0CEH}
\frac{1}{2}\int_{-1}^1\mathcal D(f)^{(0)}d\mu =  \frac{1}{2}\int_{-1}^1\mathcal D(f_0)d\mu+\BigO(\varepsilon)\,.
\end{equation}
\end{lemma}

\begin{proof}
As in the proof of Lemma~\ref{strong_lemma}, by the form of the time derivative in $\mathcal D(f)^{(i)}$ in~\eqref{Di_Chapman_Enskog}, for~\eqref{D01CEH}, we need to prove 
$$
\frac{1}{2}\int_{-1}^1\left(\mathcal D(f)^{(i)}-
\mathcal D(f_i)\right)d\mu=
\frac{1}{2}\int_{-1}^1 \left(\sum_{l=0}^i\frac{\partial_{i-l}f_{l}}{c\partial t}
-\frac{\partial f_i}{c\partial t}\right)d\mu=\mathcal{O}(\varepsilon^2)\,,
$$
i.e., after elimination of the $\varepsilon$-independent constant $c$,
\begin{equation*}
\sum_{l=0}^i\frac{\partial_{i-l}\beta_{l}}{\partial t}
=\frac{\partial \beta_i}{\partial t}+\mathcal{O}(\varepsilon^2)\,,
\end{equation*}
for $i=0,1$, which means $\frac{\partial_0 \beta_0}{\partial t} = \frac{\partial \beta_0}{\partial t}+\BigO(\varepsilon^2)$ for $i=0$ and $\frac{\partial_1 \beta_0}{\partial t} + \frac{\partial_0 \beta_1}{\partial t}= \frac{\partial \beta_1}{\partial t}+\BigO(\varepsilon^2)$ for $i=1$.

By \eqref{partial_i} we calculate for $i=0$
\begin{equation*}
\frac{\partial_0 \beta_0}{\partial t} = \frac{\partial \beta_0}{\partial t}\frac{\partial \beta_0}{\partial \beta} = \frac{\partial \beta_0}{\partial t}\left(\frac{\partial \beta}{\partial \beta}-\frac{\partial \beta^s}{\partial \beta}\right)  =  \frac{\partial \beta_0}{\partial t}(1+\BigO(\varepsilon^2)) = \frac{\partial \beta_0}{\partial t} +\BigO(\varepsilon^2)
\end{equation*}
using the assumptions \eqref{eps2ass} on $\beta_0=\beta^t$ and $\beta^s$ as well as the fact that $\beta_0$ does not depend on the spatial derivatives of $\beta$ in Enskog's postulate~\eqref{f_dep_beta}. 
With the same arguments, also for $i=1$, we conclude
\begin{multline*}
\frac{\partial_1 \beta_0}{\partial t} + \frac{\partial_0 \beta_1}{\partial t} = \frac{\partial \beta_1}{\partial t}\frac{\partial \beta_0}{\partial \beta} + \frac{\partial \beta_0}{\partial t}\frac{\partial \beta_1}{\partial \beta} = \frac{\partial \beta_1}{\partial t}\left(\frac{\partial \beta}{\partial \beta}-\frac{\partial \beta^s}{\partial \beta}\right)+\frac{\partial \beta_0}{\partial \beta}\BigO(\varepsilon^2) \\[2mm]
= \frac{\partial \beta_1}{\partial t}(1+\BigO(\varepsilon^2))+\BigO(\varepsilon^2) = \frac{\partial \beta_1}{\partial t}+\BigO(\varepsilon^2)
\end{multline*}
as required.

The derivation of \eqref{D0CEH} by \eqref{eps1ass} is the same as the derivation of \eqref{D01CEH} by \eqref{eps2ass} for $i=0$ if we replace $\BigO(\varepsilon^2)$ above by $\BigO(\varepsilon)$.
\end{proof}

Using the two preceding lemmas, we can write the second order approximation of $s$ as
\begin{multline*}
s=\frac{1}{2}\int_{-1}^{1}\left(\sum_{i=1}^\infty \varepsilon^{(i-1)} \mathcal D(f)^{(i-1)} \right)d\mu =\frac{1}{2}\int_{-1}^{1}\left(\sum_{i=1}^2 \varepsilon^{(i-1)} \mathcal D(f)^{(i-1)} \right)d\mu  +\BigO(\varepsilon^2)\\[2mm] 
= \frac{1}{2}\int_{-1}^{1} \mathcal D(f_0 + \varepsilon f_1) d\mu +\BigO(\varepsilon^2)\,.
\end{multline*}
With this approximation in \eqref{sequ0}, the equation we want to solve becomes
\begin{equation}
\frac{1}{2}\int_{-1}^{1} \mathcal D(f_0 + \varepsilon f_1) d\mu  = \frac{\bar j}{\varepsilon}- \sum_{i=0}^\infty \left(\frac{\bar{\tilde\chi}}{2\varepsilon}\int_{-1}^1 \varepsilon^{i} f_i d\mu \right) +\BigO(\varepsilon^2)\,,
\label{sequ}
\end{equation}
which is the sum of the first three equations of the hierarchy \eqref{CEHie_int_1}, \eqref{CEHie_int_2} in angularly integrated form, multiplied by $\varepsilon^{(i-1)}$, $i=0,1,2$, respectively, up to the order $\BigO(\varepsilon^2)$.

For further treatment of the left hand side of this equation, the following decomposition of the operator $\mathcal{D}$ will be helpful.

\begin{definition}
With the notation as in Subsection~\ref{fully_coupled} we define 
the operators
\begin{equation}
\label{D+}
 \mathcal D^+(f) := \frac{d f}{c dt}  + \left[ \mu \left( \frac{d \ln \rho}{cd t} + \frac{3v}{cr} \right) \right] (1-\mu^2)\frac{\partial f}{\partial \mu} + \left[ \mu^2 \left( \frac{d \ln \rho}{cd t} + \frac{3v}{cr} \right) - \frac{v}{cr}\right] \omega \frac{\partial f}{\partial \omega}
\end{equation}
and
 \begin{equation}
 \label{D-}
 \mathcal D^-(f) := \mu \frac{\partial f}{\partial r} + \frac{1}{r}(1-\mu^2)\frac{\partial f}{\partial \mu}\,.
 \end{equation}
\end{definition}

These operators decompose the operator $\mathcal{D}=\mathcal{D}^++\mathcal{D}^-$ into its symmetric and antisymmetric part with respect to $\mu$. In addition, $\mathcal{D}^+$ contains all time-dependent coefficient functions. This will be used in the next subsection in Theorem~\ref{weak_theorem_H_time} and Theorem~\ref{Free_thm}.

The following result is the fundamental lemma of this and the next subsection.

\begin{lemma}
\label{Liebend_approx}
Let $\mathcal C(f)$ be given by \eqref{coll} with $\tilde\chi\neq 0$, $\tilde\chi+\phi_0\neq 0$ as well as $\tilde\chi+\phi_0-\phi_1\neq 0$. Then the first hierarchy equation \eqref{CEHie1} has the unique solution
\begin{equation*}
f_0=\frac{j}{\tilde\chi}\,,
\end{equation*} 
which is isotropic and leads to
\begin{equation}
\label{Df0}
\mathcal D(f)^{(0)}  = \mathcal D(f_0)
\end{equation}
in the second equation \eqref{CEHie2} for $i=1$ in the hierarchy, making this equation formally the same as in the Hilbert expansion. For the solution of this equation we obtain
\begin{equation}
\label{eps_f1}
\varepsilon f_1 =\frac{-1}{\tilde\chi +\phi_0}\left(\mathcal D(f_0) + \frac{\phi_0}{\tilde\chi}\frac{1}{2}\int_{-1}^1\mathcal D(f_0)d\mu + 3\mu\phi_1 \lambda\frac{1}{2}\int_{-1}^1\mathcal D(f_0)\mu d\mu\right)\,.
\end{equation}
Furthermore, we have
\begin{equation}
\label{intDf0}
\frac{1}{2}\int_{-1}^1 \mathcal D(f_0) d\mu=\frac{1}{2}\int_{-1}^1 \mathcal D^+(f_0) d\mu=\frac{d \beta_0}{c dt} +  \frac{1}{3} \frac{d \ln \rho}{cd t}  \omega \frac{\partial \beta_0}{\partial \omega}
\end{equation}
and
\begin{multline}
\label{intDeps_f1} 
\frac{1}{2}\int_{-1}^1 \mathcal D(\varepsilon f_1) d\mu\doteq\frac{1}{2}\int_{-1}^1 \mathcal D^-(\varepsilon f_1) d\mu
 = -\frac{1}{2}\int_{-1}^1 \mathcal D^-\left(\lambda\mathcal D^- (f_0)\right)d\mu
= -\frac{1}{r^2}\frac{\partial }{\partial r}\left( r^2\frac{\lambda}{3}\frac{\partial \beta_0}{\partial r}\right)\,,
\end{multline}
i.e.,
\begin{equation}
\label{leading_order}
\frac{1}{2}\int_{-1}^{1} \mathcal D(f_0 + \varepsilon f_1) d\mu \doteq \frac{d\beta_0}{cd t}+\frac{1}{3}\frac{d \ln\rho}{cd t}\omega\frac{\partial\beta_0}{\partial \omega} - \frac{1}{r^2}\frac{\partial}{\partial r}\left(r^2\frac{\lambda}{3}\frac{\partial\beta_0}{\partial r}\right)
\end{equation}
in which $\doteq$ denotes the identity up to $\mathcal{D}^+\big(\mathcal{D}^+(\cdot)\big)$ terms if we use the decomposition of $\mathcal{D}=\mathcal{D}^++\mathcal{D}^-$ in \eqref{eps_f1} and \eqref{intDeps_f1}.
\end{lemma}


\begin{remark}
The expressions for $f_0$ and $\varepsilon f_1$ are obtained by taking the zeroth as well as the first angular moment of the first two equations of the hierarchy \eqref{CEHie1} and \eqref{CEHie2}, respectively, and inserting the obtained expressions for $\frac{1}{2}\int_{-1}^1f_id\mu$ and $\frac{1}{2}\int_{-1}^1f_i\mu d\mu$ back into \eqref{CEHie1} for $i=0$ and \eqref{CEHie2} for $i=1$, respectively. Here one uses that the particular form of the truncated Legendre expansion of the collision integral $\mathcal C(f)$ given in equation \eqref{coll} is given in terms of $f$ and the zeroth and first angular moment of $f$. The calculations are similar as performed in~\cite[App.\;A]{LiebendoerferEtAl09big} and in~\cite[Sec.~2.4]{BeFrGaLiMiVa12_esaimbig}.\\
\indent The arguments of how to obtain \eqref{leading_order} are also sketched in \cite[App.\;A]{LiebendoerferEtAl09big} and are given in more detail in~\cite[Sec.~2.4]{BeFrGaLiMiVa12_esaimbig}. However, for proper readability of this paper, we need to give a comprehensive presentation of the proof of this fundamental lemma here, too.\\
In this Lemma, and in the Theorems on the diffusion and the reaction limits to come, $f_0=j/\tilde\chi$ turns out to be isotropic, so that we obtain the identity
$$
f_0=\frac{1}{2}\int_{-1}^1 f_0 d\mu=\beta_0
$$
in which, as in \eqref{f_t_is_beta_t}, the first equation is a slight abuse of notation because its left hand side is constant with respect to $\mu$ whereas, a priori by integration, the right hand side does not even contain a $\mu$-dependency. 
\end{remark}

\begin{proof}[Proof of Lemma \ref{Liebend_approx}]
Using \eqref{collision_vanishes} and $\tilde\chi\neq 0$, the zeroth moment of the first equation \eqref{CEHie1} in the hierarchy reads
\begin{equation}
\label{f0=b0}
0=j-\tilde \chi\frac{1}{2}\int_{-1}^{1}f_0d\mu\,,\quad\mbox{i.e.,}\quad\frac{1}{2}\int_{-1}^{1}f_0d\mu=\frac{j}{\tilde\chi}\,.
\end{equation}
Using \eqref{coll}, the first moment of \eqref{CEHie1} can be rewritten as
\begin{eqnarray*}
0&=&\frac{1}{2}\int_{-1}^1 \Big(j-\tilde \chi f_0 + \mathcal C(f_0)\Big)\mu d\mu\\[2mm]
\Leftrightarrow\quad 0&=& -\left(\tilde\chi-\phi_0+\frac{3\phi_1}{2}\int_{-1}^1\mu^2d\mu\right)\frac{1}{2}\int_{-1}^1 f_0 \mu d\mu \\[2mm]
\Leftrightarrow\quad 0&=&\frac{1}{2}\int_{-1}^1 f_0 \mu d\mu
\end{eqnarray*}
if $\tilde\chi+\phi_0-\phi_1\neq 0$. Together with \eqref{f0=b0} and \eqref{coll} in \eqref{CEHie1}, this result leads to
\begin{equation*}
0=j-\tilde\chi f_0 -\phi_0 f_0+\phi_0\frac{1}{2}\int_{-1}^{1}f_0d\mu+3\mu\phi_1\frac{1}{2}\int_{-1}^1 f_0 \mu d\mu=j-(\tilde\chi+\phi_0) f_0+\phi_0\frac{j}{\tilde\chi} \,,
\end{equation*}
i.e.,  $f_0=j/\tilde\chi$ if $\tilde\chi\neq 0$ and $\tilde\chi+\phi_0\neq 0$.

In order to derive \eqref{eps_f1} we first prove \eqref{Df0}
for which, by \eqref{Di_Chapman_Enskog}, we have to show
\begin{equation}
\label{Df00}
\frac{\partial_0 f_0}{\partial t}=\frac{\partial f_0}{\partial t}\,.
\end{equation}
This property follows from isotropy of $f_0=j/\tilde\chi$ which holds since both $j$ and $\tilde\chi$ do not depend on~$\mu$. Concretely, since $f_0$ is isotropic, we immediately have $f_0=\beta_0$, and since $\beta$ and, therefore,~$\beta_0$, too, do not depend on the spatial derivatives of $\beta$ by postulate~\eqref{f_dep_beta}, we obtain \eqref{Df00} from \eqref{partial_i}.

Using \eqref{Df0} and \eqref{collision_vanishes}, the zeroth moment of the second equation \eqref{CEHie2} for $i=1$ in the hierarchy reads
\begin{equation}
\label{zeroth_f1}
\frac{1}{2}\int_{-1}^1 \mathcal D(f_0) d\mu=-\tilde\chi\frac{1}{2}\int_{-1}^1 \eps f_1 d\mu\,.
\end{equation}
Again, by \eqref{Df0} and the special form of the collision integral \eqref{coll}, the first moment of \eqref{CEHie2} for $i=1$ can be written as
\begin{eqnarray}
\frac{1}{2}\int_{-1}^1 \mathcal D(f_0)\mu d\mu &=& -(\tilde\chi+\phi_0)\frac{1}{2}\int_{-1}^1 \eps f_1 \mu d\mu+\phi_1 \frac{3}{2}\int_{-1}^1\mu^2 d\mu\; \frac{1}{2}\int_{-1}^1 \eps f_1\mu d\mu\nonumber\\[2mm]
\Leftrightarrow\quad\frac{1}{2}\int_{-1}^1 \mathcal D(f_0)\mu d\mu &=& -(\tilde\chi+\phi_0-\phi_1)\frac{1}{2}\int_{-1}^1 \eps f_1 \mu d\mu\,.\label{first_f1}
\end{eqnarray}
Now, inserting the expressions \eqref{zeroth_f1} and \eqref{first_f1} for the zeroth and first moment of $f_1$ in terms of the zeroth and first moment of $\mathcal{D}(f_0)$ into \eqref{CEHie2} for $i=1$ with the collision integral given by \eqref{coll}, we obtain
\begin{equation}
\label{Df0-f1}
\mathcal{D}(f_0)=-(\tilde\chi+\phi_0)\eps f_1-\frac{\phi_0}{\tilde\chi}\frac{1}{2}\int_{-1}^1 \mathcal D(f_0) d\mu-\frac{3\mu\phi_1}{\tilde\chi+\phi_0-\phi_1}\frac{1}{2}\int_{-1}^1 \mathcal D(f_0)\mu d\mu\,.
\end{equation}
By the definition \eqref{meanfreepath} of the mean free path $\lambda$, this equation is equivalent to \eqref{eps_f1}.



Now we turn to the proof of \eqref{intDf0} and \eqref{intDeps_f1}. Using the decomposition $\mathcal D=\mathcal D^++\mathcal D^-$, the first equality in \eqref{intDf0} follows immediately from $\frac{1}{2}\int_{-1}^1 \mathcal D^-(f_0) d\mu = 0$ since $\mathcal D^-$ is antisymmetric in $\mu$ and $f_0$ does not depend on $\mu$. In order to prove the second equality in \eqref{intDf0} we calculate
\begin{align*}
&\frac{1}{2}\int_{-1}^1 \mathcal D^+(f_0) d\mu =\nonumber\\[3mm]
&  \frac{1}{2}\int_{-1}^1\left[ \frac{d f_0}{c dt}  + \left[ \mu \left( \frac{d \ln \rho}{cd t} + \frac{3v}{cr} \right) \right] (1-\mu^2)\frac{\partial f_0}{\partial \mu} + \left[ \mu^2 \left( \frac{d \ln \rho}{cd t} + \frac{3v}{cr} \right) - \frac{v}{cr}\right] \omega \frac{\partial f_0}{\partial \omega}\right] d\mu\nonumber\\[3mm]
&= \frac{d \beta_0}{c dt} + \left[ \frac{1}{3} \left( \frac{d \ln \rho}{cd t} + \frac{3v}{cr} \right) -\frac{v}{cr} \right]\omega \frac{\partial \beta_0}{\partial \omega}\nonumber\\[3mm]
&= \frac{d \beta_0}{c dt} +  \frac{1}{3} \frac{d \ln \rho}{cd t}  \omega \frac{\partial \beta_0}{\partial \omega}\,.
\end{align*}
Here, the second equality is obtained by formally interchanging integration and differentiation, considering $\beta_0=\frac{1}{2}\int_{-1}^1f_0 d\mu$ in the first and third summand in the integral, using $\frac{1}{2}\int_{-1}^1 \mu^2d\mu=\frac{1}{3}$ and the isotropy of $f_0$ for the latter and observing that the second summand vanishes completely since $f_0$ does not depend on $\mu$.

For the approximation ``$\doteq$'' in \eqref{intDeps_f1} we also use the decomposition $\mathcal D=\mathcal D^++\mathcal D^-$ and first consider the $\mathcal D^+$ part which, by inserting \eqref{eps_f1}, reads
\begin{multline*}
\frac{1}{2}\int_{-1}^1 \mathcal D^+(\varepsilon f_1) d\mu
=
\frac{1}{2}\int_{-1}^1 \mathcal D^+\Bigg(\frac{-1}{\tilde \chi+\phi_0}\Bigg[ \mathcal D^+(f_{0})+\mathcal D^-(f_0)
+\frac{\phi_0}{\tilde \chi}\frac{1}{2}\int_{-1}^{1}\big(\mathcal D^+(f_0)+\mathcal D^-(f_0)\big)d\mu\\[1mm]
+3\mu\phi_1\lambda\frac{1}{2}\int_{-1}^1 \big(\mathcal D^+(f_0)+\mathcal D^-(f_0)\big) \mu d\mu \Bigg]\Bigg) d\mu.
\end{multline*}
Since $f_0$ does not depend on $\mu$, the term $\mathcal D^+(f_0)\mu$ in the second inner integral of this expression is an antisymmetric polynomial in $\mu$ and thus vanishes after integration. By the same reasoning, the term $\mathcal D^-(f_0)\mu$ in the second inner integral is a symmetric polynomial in $\mu$ so that angular integration gives an expression that no longer depends on $\mu$. Consequently, the last summand in the brackets is linear in $\mu$ so that the application of the operator $\mathcal D^+\big(\frac{-1}{\tilde \chi+\phi_0}(\cdot)\big)$ to this expression is an antisymmetric polynomial in $\mu$ that vanishes after the outer integration. By the same reasoning, the second summand in the brackets vanishes after the application of this operator and the outer integration. With these arguments and $\frac{1}{2}\int_{-1}^1 \mathcal D^-(f_0) d\mu = 0$ as seen above we obtain the simplified equation
\begin{equation}
\label{D++terms}
\frac{1}{2}\int_{-1}^1 \mathcal D^+(\varepsilon f_1) d\mu
=
\frac{1}{2}\int_{-1}^1 \mathcal D^+\Bigg(\frac{-1}{\tilde \chi+\phi_0}\Bigg[ \mathcal D^+(f_{0})
+\frac{\phi_0}{\tilde \chi}\frac{1}{2}\int_{-1}^{1}\mathcal D^+(f_0)d\mu\Bigg]\Bigg) d\mu\, ,
\end{equation}
in which the right hand side only contains terms that undergo the application of the operator $\mathcal D^+$ twice. Consequently, this part contains all $\mathcal D^+\big(\mathcal D^+(\cdot)\big)$ terms that are neglected in the approximation ``$\doteq$'' in \eqref{intDeps_f1} and~\eqref{leading_order}.

Second, the $\mathcal D^-$ part in the approximation ``$\doteq$'' in \eqref{intDeps_f1} reads
\begin{multline*}
\frac{1}{2}\int_{-1}^1 \mathcal D^-(\varepsilon f_1) d\mu
=
\frac{1}{2}\int_{-1}^1 \mathcal D^-\Bigg(\frac{-1}{\tilde \chi+\phi_0}\Bigg[ \mathcal D^+(f_{0})+\mathcal D^-(f_0)
+\frac{\phi_0}{\tilde \chi}\frac{1}{2}\int_{-1}^{1}\big(\mathcal D^+(f_0)+\mathcal D^-(f_0)\big)d\mu\\[1mm]
+3\mu\phi_1\lambda\frac{1}{2}\int_{-1}^1 \big(\mathcal D^+(f_0)+\mathcal D^-(f_0)\big) \mu d\mu \Bigg]\Bigg) d\mu.
\end{multline*}
As for the $\mathcal D^+$ part, the term $\mathcal D^+(f_0)\mu$ vanishes after integration in the second inner integral. Since $f_0$ is independent of $\mu$, the term $\mathcal D^+(f_0)$ in the first inner integral is a symmetric polynomial in $\mu$. Therefore, since $\frac{1}{2}\int_{-1}^1 \mathcal D^-(f_0) d\mu = 0\,$, the first inner integral does no longer depend on $\mu$ so that the application of the operator $\mathcal D^-\big(\frac{-1}{\tilde \chi+\phi_0}(\cdot)\big)$ to it is linear in $\mu$ and vanishes after the outer integration. The same holds for the first term $\mathcal D^+(f_0)$ in the brackets which is a symmetric polynomial in $\mu$ so that the application of $\mathcal D^-\big(\frac{-1}{\tilde \chi+\phi_0}(\cdot)\big)$ to it gives an antisymmetric polynomial in $\mu$ that vanishes by the outer integration. Consequently, we obtain the equation
\begin{equation*}
\frac{1}{2}\int_{-1}^1 \mathcal D^-(\varepsilon f_1) d\mu=
\frac{1}{2}\int_{-1}^1 \mathcal D^-\Bigg(\frac{-1}{\tilde \chi+\phi_0}\Bigg[ \mathcal D^-(f_{0})
+3\mu\phi_1\lambda\frac{1}{2}\int_{-1}^1 \mathcal D^-(f_0) \mu d\mu\Bigg]\Bigg) d\mu\, ,
\end{equation*}
that we can further simplify by observing 
\begin{equation}
\label{D-f0}
\mathcal D^-(f_0)=\mu\frac{\partial f_0}{\partial r}\,,
\end{equation}
which leads to
$$ 
3\mu\phi_1\lambda\frac{1}{2}\int_{-1}^1 \mathcal D^-(f_0) \mu d\mu=3\mu\phi_1\lambda\frac{1}{2}\int_{-1}^1 \mu^2\frac{\partial f_0}{\partial r} d\mu=\phi_1\lambda\mu\frac{\partial f_0}{\partial r}=\phi_1\lambda\mathcal D^-(f_0)\,,
$$
so that, with \eqref{meanfreepath}, we can conclude
\begin{equation*}
\frac{1}{2}\int_{-1}^1 \mathcal D^-(\varepsilon f_1) d\mu=
\frac{1}{2}\int_{-1}^1 \mathcal D^-\Bigg(\frac{-1}{\tilde \chi+\phi_0}\Big[ \big( 1
 +\phi_1\lambda \big)\mathcal D^-(f_0)\Big]\Bigg) d\mu
 = \frac{1}{2}\int_{-1}^1 \mathcal D^-\left(-\lambda\mathcal D^- (f_0)\right)d\mu\,.
\end{equation*}

In order to prove the last equality in \eqref{intDeps_f1} we take \eqref{D-f0} into account and calculate
\begin{align*}
\frac{1}{2}\int_{-1}^1 \mathcal D^-\left(\lambda\mathcal D^- (f_0)\right)d\mu
&= \frac{1}{2}\int_{-1}^1 \mathcal D^-\left(\lambda\mu\frac{\partial f_0}{\partial r}\right)d\mu\nonumber\\[3mm]
&= \frac{1}{2}\int_{-1}^1\left[ \mu \frac{\partial}{\partial r}  \left(\lambda\mu\frac{\partial f_0}{\partial r}\right)  + \frac{1}{r}(1-\mu^2)\frac{\partial }{\partial \mu} \left(\lambda\mu\frac{\partial f_0}{\partial r}\right) \right]     d\mu\nonumber\\[3mm]
&= \frac{1}{3} \frac{\partial}{\partial r}\left(  \lambda\frac{\partial \beta_0}{\partial r}\right) +  \frac{1}{2}\int_{-1}^1  \frac{1}{r}(1-\mu^2) \lambda\frac{\partial f_0}{\partial r} d\mu\nonumber\\[3mm]
&= \frac{1}{3} \frac{\partial}{\partial r}\left(  \lambda\frac{\partial \beta_0}{\partial r}\right) + \frac{2}{3}\frac{1}{r}\lambda \frac{\partial \beta_0}{\partial r}\nonumber\\[3mm]
&= \frac{1}{r^2}\frac{\partial }{\partial r}\left( r^2\frac{\lambda}{3}\frac{\partial \beta_0}{\partial r}\right).
\end{align*}
For the third and fourth equality, we used again $\frac{1}{2}\int_{-1}^1 \mu^2d\mu=\frac{1}{3}$, the isotropy of $f_0$ and $\lambda$ as well as $\beta_0=\frac{1}{2}\int_{-1}^1f_0 d\mu$. The last equality follows from the product rule. 

To finish the proof we note that \eqref{leading_order} follows immediately from the sum of \eqref{intDf0} and \eqref{intDeps_f1}.
\end{proof}

\begin{remark}
\label{D+D+} 
The approximation \eqref{leading_order} of $s$ in the sense of Lemma \ref{Liebend_approx} is called ``leading order approximation'' in \cite{LiebendoerferEtAl09big} since the ``leading $\frac{1}{2}\int_{-1}^1\mathcal{D}^+(f_0)d\mu$ term'' is included on the right hand side. Note that if we introduce an additional scaling for the time variable as in Theorem~\ref{weak_theorem_H_time}, then the $\mathcal D^+\big(\mathcal D^+(\cdot)\big)$ terms \eqref{D++terms} are naturally of order $\varepsilon^2$ so that \eqref{leading_order} is replaced by the sum of \eqref{intDf0} and \eqref{intDeps_f1} in which the first ``$\doteq$''-equation is skipped. This sum is then an approximation of $s$ up to the order~$\BigO(\varepsilon^2)$.
\end{remark}

We can now collect the results and formulate our main approximation theorems.

\begin{theorem} 
\label{strong_theorem}
Let $\mathcal C(f)$ be given by \eqref{coll} with $\tilde\chi\neq 0$, $\tilde\chi+\phi_0\neq 0$ as well as $\tilde\chi+\phi_0-\phi_1\neq 0$. We assume Enskog's postulate~\eqref{f_dep_beta}. Then, with the strong Chapman--Enskog hypothesis in Lemma \ref{strong_lemma}, 
the diffusion-reaction-type equation
\begin{equation*}
\frac{d\beta}{cd t}+\frac{1}{3}\frac{d \ln\rho}{cd t}\omega\frac{\partial\beta}{\partial \omega} - \frac{1}{r^2}\frac{\partial}{\partial r}\left(r^2\frac{\lambda}{3}\frac{\partial\beta}{\partial r}\right) =  j-\tilde\chi \beta
\end{equation*}
is an approximation of the angular mean of Boltzmann's equation \eqref{Boltzmann_eps} of order $\BigO(\varepsilon^2)$ up to $\mathcal{D}^+\big(\mathcal{D}^+(\cdot)\big)$ terms in the sense of Lemma~\ref{Liebend_approx}.
\end{theorem}

\begin{theorem}[$\Sigma$ in the diffusion limit I] 
\label{weak_theorem}
Let $\mathcal C(f)$ be given by \eqref{coll} with $\tilde\chi\neq 0$, $\tilde\chi+\phi_0\neq 0$ and $\tilde\chi+\phi_0-\phi_1\neq 0$. We consider the decomposition
$f = f^t + f^s$ inducing $\beta = \beta^t +\beta^s$ with the angular means $\beta^t=\frac{1}{2}\int_{-1}^1 f^t d\mu$ and $\beta^s=\frac{1}{2}\int_{-1}^1 f^s d\mu$, and we assume Enskog's postulate~\eqref{f_dep_beta}. Then, with the weakened Chapman--Enskog hypothesis and with the asymptotic properties \eqref{eps2ass} of $\beta^s$ and $\beta_1$ in Lemma \ref{weak_lemma}, 
the diffusion-reaction-type equation
\begin{equation}
\label{diff_react_weak}
\frac{d\beta^t}{cd t}+\frac{1}{3}\frac{d \ln\rho}{cd t}\omega\frac{\partial\beta^t}{\partial \omega} - \frac{1}{r^2}\frac{\partial}{\partial r}\left(r^2\frac{\lambda}{3}\frac{\partial\beta^t}{\partial r}\right) =  j-\tilde\chi (\beta^t+\beta^s)
\end{equation}
is an approximation of the angular mean of Boltzmann's equation \eqref{Boltzmann_eps} of order $\BigO(\varepsilon^2)$ up to $\mathcal{D}^+\big(\mathcal{D}^+(\cdot)\big)$ terms in the sense of Lemma~\ref{Liebend_approx}. In the same sense 
\begin{equation}
\label{diff_source}
\Sigma_{\rm{diff}} := - \frac{1}{r^2}\frac{\partial}{\partial r}\left(r^2\frac{\lambda}{3}\frac{\partial\beta^t}{\partial r}\right) +\tilde\chi \beta^s
\end{equation}
is an approximation of the (angular mean of the) diffusion source $\Sigma$ in~\eqref{eq:trapped}. 
\end{theorem}

\begin{proof}[Proof of Theorems \ref{strong_theorem} and \ref{weak_theorem}]
First we recall that the asymptotic relation \eqref{sequ} is a consequence of both Lemma~\ref{strong_lemma} and relation~\eqref{D01CEH} in Lemma~\ref{weak_lemma} which follows from \eqref{eps2ass}. We insert the approximation \eqref{leading_order} into \eqref{sequ}. Then the equation in Theorem~\ref{strong_theorem} follows from $\beta_0 = \beta$. For Theorem~\ref{weak_theorem} we derive equation \eqref{diff_react_weak} from  $\beta_0 = \beta^t$ and $\beta=\beta^t+\beta^s$. The second statement of Theorem~\ref{weak_theorem} follows from the comparison of \eqref{diff_react_weak} with \eqref{eq:trapped} taking into account \eqref{f_t_is_beta_t}.
\end{proof}

In summary, Theorem \ref{weak_theorem} states that the diffusion source~\eqref{Sigma_diff} in the IDSA \cite[App.\;A]{LiebendoerferEtAl09big} follows from a Chapman--Enskog expansion of second order with the weakened Chapman--Enskog hypothesis and the asymptotic properties \eqref{eps2ass} of $\beta^s$ and $\beta_1$ in Lemma~\ref{weak_lemma} up to $\mathcal{D}^+\big(\mathcal{D}^+(\cdot)\big)$ terms in the sense of Lemma~\ref{Liebend_approx}.

\begin{remark}
\label{ftnotneciso}
There is a slight but important subtlety to the conditions imposed on the decomposition $f=f^t+f^s$ in Theorem~\ref{weak_theorem} and in Lemma~\ref{weak_lemma}. In contrast to the same decomposition in Section~\ref{IDSA}, we do not assume $f^t$ to be isotropic. Instead, we only impose conditions on the angular means $\beta^t$ and~$\beta^s$, which reflects the fact that we only make assertions on the angularly integrated Boltzmann equation anyway. In particular, we set $\beta^t=\beta_0$, and we know $\beta_0=j/\tilde\chi$ from Lemma~\ref{Liebend_approx}. Consequently, if $f^t$ is isotropic, it is necessarily identical to the equilibrium function $f_0=j/\tilde\chi$. This is mathematically not forbidden but physically unrealistic and only satisfied in special cases, e.g., in the center of a supernova after long times. Therefore, the assumption in the IDSA that $f^t$ should be isotropic, i.e,~\eqref{f_t_is_beta_t}, seems inappropriate. As Theorem~\ref{weak_theorem} shows, we do not need it in order to establish a diffusion-reaction-type equation as an approximation of the angular mean of Boltzmann's equation in the diffusion limit and in order to define the same diffusion source as in the~IDSA.

On the other hand, the trapped particle equation~\eqref{eq:trapped} cannot be derived from~\eqref{decomposedtrapped} if isotropy of $f^t$ is not assumed since then additional terms occur on the left hand side coming from the non-isotropic part of~$f^t$. Indeed, in Theorem~\ref{weak_theorem}, it is the diffusion term in~\eqref{diff_source} that actually accounts for the non-isotropic part of $f\sim f_0+\eps f_1$, which is $\eps f_1$ in leading order. If $\eps f_1$ is isotropic, we have $\frac{1}{2}\mathcal{D}^-(\eps f_1)d\mu=0$ since $\mathcal{D}^-$ is antisymmetric in $\mu$ which, by~\eqref{intDeps_f1}, entails a vanishing diffusion~term.
\end{remark}

We close this section by stating the corresponding results for the approximation of the total interaction rate $s$ up to first order only, i.e., if we consider the first two hierarchy equations \eqref{CEHie1} and \eqref{CEHie2} for $i=1$ instead of the first three equations as done above for Theorems \ref{strong_theorem} and~\ref{weak_theorem}. It turns out that in this case we obtain a reaction-type equation instead of a diffusion-reaction-type equation.

\begin{theorem} 
\label{strong_theorem_react}
Let $\mathcal C(f)$ be given by \eqref{coll} with $\tilde\chi\neq 0$, $\tilde\chi+\phi_0\neq 0$ and $\tilde\chi+\phi_0-\phi_1\neq 0$. We assume Enskog's postulate~\eqref{f_dep_beta}. Then, with the strong Chapman--Enskog hypothesis in Lemma \ref{strong_lemma}, 
the reaction-type equation
\begin{equation*}
\frac{d\beta}{cd t}+\frac{1}{3}\frac{d \ln\rho}{cd t}\omega\frac{\partial\beta}{\partial \omega} =  j-\tilde\chi \beta
\end{equation*}
is an approximation of the angular mean of Boltzmann's equation \eqref{Boltzmann_eps} of order $\BigO(\varepsilon)$.
\end{theorem}

\begin{theorem}[$\Sigma$ in the reaction limit I] 
\label{weak_theorem_react}
Let $\mathcal C(f)$ be given by \eqref{coll} with $\tilde\chi\neq 0$, $\tilde\chi+\phi_0\neq 0$ and $\tilde\chi+\phi_0-\phi_1\neq 0$. We consider the decomposition
$f = f^t + f^s$ inducing $\beta = \beta^t +\beta^s$ with the angular means $\beta^t=\frac{1}{2}\int_{-1}^1 f^t d\mu$ and $\beta^s=\frac{1}{2}\int_{-1}^1 f^s d\mu$, and we assume Enskog's postulate~\eqref{f_dep_beta}. Then, with the weakened Chapman--Enskog hypothesis and with the asymptotic properties \eqref{eps1ass} of $\beta^s$ and $\beta_1$ in Lemma \ref{weak_lemma}, 
the reaction-type equation
\begin{equation}
\label{react_weak}
\frac{d\beta^t}{cd t}+\frac{1}{3}\frac{d \ln\rho}{cd t}\omega\frac{\partial\beta^t}{\partial \omega} =  j-\tilde\chi (\beta^t+\beta^s)
\end{equation}
is an approximation of the angular mean of Boltzmann's equation \eqref{Boltzmann_eps} of order $\BigO(\varepsilon)$. Accordingly,
\begin{equation*}
\Sigma_{\rm{reac}} := \tilde\chi \beta^s
\end{equation*}
is an approximation of order $\BigO(\varepsilon)$ of the (angular mean of the) source term $\Sigma$ in~\eqref{eq:trapped}. If we additionally suppose $\beta_1 = \BigO(\varepsilon)$, then we have $\beta^s =\BigO(\varepsilon^2)$ and the reaction limit proposed in Subsection~\ref{Sigma_reaction},
$$
\tilde\Sigma_{\rm{reac}} := 0\,,
$$ 
is an approximation of order $\BigO(\varepsilon)$ of $\Sigma$ in~\eqref{eq:trapped}.
\end{theorem}

\begin{proof}[Proof of Theorems \ref{strong_theorem_react} and \ref{weak_theorem_react}]
We first note that the asymptotic relation 
\begin{equation}
\label{sequ_react}
\frac{1}{2}\int_{-1}^{1} \mathcal D(f_0) d\mu  = \frac{\bar j}{\varepsilon}- \sum_{i=0}^\infty \left(\frac{\bar{\tilde\chi}}{2\varepsilon}\int_{-1}^1 \varepsilon^{i} f_i d\mu \right) +\BigO(\varepsilon)
\end{equation}
is a consequence of both Lemma~\ref{strong_lemma} and relation~\eqref{D0CEH} in Lemma~\ref{weak_lemma} which follows from \eqref{eps1ass}. We insert the approximation \eqref{intDf0} into \eqref{sequ_react}. Then the equation in Theorem~\ref{strong_theorem_react} follows from $\beta_0 = \beta$. For Theorem~\ref{weak_theorem_react} we derive equation \eqref{react_weak} from  $\beta_0 = \beta^t$ and $\beta=\beta^t+\beta^s$. The second statement of Theorem~\ref{weak_theorem_react} follows from the comparison of \eqref{react_weak} with \eqref{eq:trapped} taking into account \eqref{f_t_is_beta_t}. The last assertion follows immediately from $\beta^s=\beta-\beta^t=\eps\beta_1+\eps^2\beta_2+\ldots$, the hypothesis $\tilde\chi=\BigO(\varepsilon^{-1})$ and the same comparison of equations.
\end{proof}

Summarizing, Theorem \ref{weak_theorem_react} states that the lower bound $0$ for the diffusion source $\Sigma$ in the IDSA \cite[App.\;A]{LiebendoerferEtAl09big} follows from a Chapman--Enskog expansion of first order with the weakened Chapman--Enskog hypothesis and the additional asymptotic properties \eqref{eps1ass} of $\beta^s$ and $\beta_1$ as stated in Lemma \ref{weak_lemma} and in Theorem \ref{weak_theorem_react}. Without the latter assumption on $\beta_1$ we only have $\beta^s=\BigO(\varepsilon)$ and the source term $\Sigma_{\rm{reac}}$ is of order $\BigO(1)$ due to $\tilde\chi=\BigO(\varepsilon^{-1})$. However, since on a physical level the free streaming component $\beta^s$ should become very small in a reaction regime, this additional assumption might be justified.

Finally, Remark~\ref{ftnotneciso} also applies to Theorem~\ref{weak_theorem_react} since here, the angular mean of the Boltzmann equation is compared to the trapped particle equation~\eqref{eq:trapped}, too. In the reaction limit, $\Sigma$ does not contain the diffusion term that represents non-isotropic parts of $f$ as indicated in Remark~\ref{ftnotneciso}. This reflects the fact that in the reaction limit it is physically more realistic for $f^t$ to be at least close to the isotropic equilibrium function $j/\tilde\chi$.


\subsection{Hilbert expansion}
\label{Hilbert}

In the preceding section, we derived the diffusion limit (Subsections~\ref{diffusion_source} and~\ref{Sigma_reaction}) and the reaction limit (Subsection~\ref{Sigma_reaction}) of Boltzmann's equation by means of Chapman--Enskog expansions. In~this section, we will use Hilbert expansions, first in order to derive the diffusion and the reaction limits 
with less assumptions on the asymptotic behavior of the components in the expansion. In fact, it turns out that these additional assumptions (as given in Lemmas~\ref{strong_lemma} or~\ref{weak_lemma}, respectively) were only needed above to obtain a Hilbert expansion setting from a Chapman--Enskog expansion ansatz. Moreover, in a second step, we can even strengthen an assertion by quite naturally deriving approximations of a certain $\eps$-order for the diffusion limit without the restrictive exception of ``up to $\mathcal{D}^+\big(\mathcal{D}^+(\cdot)\big)$ terms'' as in Lemma~\ref{Liebend_approx} and Theorems \ref{strong_theorem}--\ref{weak_theorem_react}. This will be achieved by an additional scaling of the time variable in the operator $\mathcal D$ as done, e.g., in Degond and Jin~\cite{DegondJin05big}, compare also \cite[p.\,131]{CercignaniKremer02}, \cite{BardosSantosSentis84}, 
\cite{GoudonPoupaud01} and \cite{Perthame04} and the literature cited therein. Finally, by just considering the time scaling without the scaling of reactions and collisions, we can also derive the free streaming limit of the IDSA introduced in Subsection~\ref{Sigma_free_streaming}.                                                                                                              
For an introduction to Hilbert expansions we refer to~\mbox{\cite[pp.\;109--115]{FerzigerKaper72}.} The ansatz is the same as for Chapman--Enskog expansions since it starts with a scaling of the right hand side of Boltzmann's equation as in~\eqref{Boltzmann_eps} and an expansion of the distribution function $f$ in powers of $\varepsilon$ as in~\eqref{f_expanded}. However, instead of the sophisticated operator expansion \eqref{D_Chapman_Enskog} that is based on the Enskog postulate~\eqref{f_dep_beta}, Hilbert uses linearity and the formal continuity of $\mathcal{D}$ which leads to the simple expansion
\begin{equation}
\label{D_Hilbert_1}
\mathcal D(f) = \mathcal D(f_0) +\varepsilon\mathcal D(f_1) +\varepsilon^2\mathcal D(f_2) +\dots\,.
\end{equation}
In addition, instead of the Chapman--Enskog hypothesis \eqref{beta_0=beta} and~\eqref{beta_i=0}, Hilbert assumes another condition that leads to unique solvability of the obtained hierarchy of equations which shall not be discussed here. (With regard to this topic we refer the reader to~\cite[p.~111]{FerzigerKaper72}.)

Using the expansions \eqref{f_expanded} for $f$ and \eqref{D_Hilbert_1} for $\mathcal{D}$ in Boltzmann's equation \eqref{Boltzmann_eps} we obtain the hierarchy of equations
\begin{eqnarray}
\label{H1Hie1}
0&=&j+\mathcal J(f_0)\,,\\[1mm]
\label{H1Hie2}
\mathcal D(f_{i-1}) &=& \bar{\mathcal J}(f_i)\,,\quad\quad i = 1,2, \dots\,,
\end{eqnarray}
corresponding to $\varepsilon$-powers $i-1$ for $i=0,1,\dots$. With the scaling in \eqref{Boltzmann_eps} that is applied only to the reactions and the collisions, we now derive Theorems \ref{strong_theorem}--\ref{weak_theorem_react} and, in particular, the diffusion and the reaction limit of Boltzmann's equation in the same way as based on Chapman--Enskog expansions. Remark~\ref{ftnotneciso} also applies here.

\begin{theorem} 
\label{strong_theorem_H}
Theorem \ref{strong_theorem} also holds without assuming Enskog's postulate~\eqref{f_dep_beta} and without the strong Chapman--Enskog hypothesis in Lemma~\ref{strong_lemma}.
\end{theorem}

\begin{theorem}[$\Sigma$ in the diffusion limit II] 
\label{weak_theorem_H}
Theorem \ref{weak_theorem} also holds without assuming Enskog's postulate~\eqref{f_dep_beta}, without the weakened Chapman--Enskog hypothesis in Lemma~\ref{weak_lemma} and without assuming the asymptotic properties \eqref{eps2ass} of $\beta^s$ and $\beta_1$.
\end{theorem}

\begin{theorem} 
\label{strong_theorem_react_H}
Theorem \ref{strong_theorem_react} also holds without assuming Enskog's postulate~\eqref{f_dep_beta} and without the strong Chapman--Enskog hypothesis in Lemma~\ref{strong_lemma}.
\end{theorem}

\begin{theorem}[$\Sigma$ in the reaction limit II] 
\label{weak_theorem_react_H}
\mbox{}\!\!Theorem \ref{weak_theorem_react} also holds without assuming Enskog's postulate~\eqref{f_dep_beta}, without the weakened Chapman--Enskog hypothesis in Lemma~\ref{weak_lemma} and without assuming the asymptotic property~\eqref{eps1ass} of $\beta^s$.
\end{theorem}

The proofs of Theorems \ref{strong_theorem_H}--\ref{weak_theorem_react_H} can now be based on the hierarchy of equations \eqref{H1Hie1} and \eqref{H1Hie2} according to Hilbert instead of the hierarchy of equations \eqref{CEHie1} and \eqref{CEHie2} according to Chapman and Enskog. Lemmas~\eqref{strong_lemma} and \eqref{weak_lemma} were only established to be able to use the former if one starts with the latter. Therefore, the assumptions in these lemmas are not needed here, and the other parts of the proofs are exactly the same as above since they are anyway based on the operator expansion~\eqref{D_Hilbert_1}. These arguments make it evident that for our purposes, Hilbert expansions seem to provide a more natural approach than Chapman--Enskog expansions. 

In the next result, we get rid of yet another restriction, now in the assertion of Theorem \ref{weak_theorem} or~\ref{weak_theorem_H} in which the approximations are in fact of order $\BigO(\eps^2)$ without the need to neglect any further terms. This result will quite naturally be obtained by the introduction of a time scaling 
\begin{equation}
\label{time_scaling}
t=\frac{\bar t}{\eps}
\end{equation}
in addition to the scaling of reactions and collisions in~\eqref{Boltzmann_eps}, compare~\cite{DegondJin05big} {or~\cite[p.\;131]{CercignaniKremer02}}. Concretely, we consider the scaled Boltzmann equation
\begin{equation}
\label{Boltzmann_eps_teps}
\mathcal D_\eps(f)= \varepsilon^{-1} \left(\bar{j}+ \bar{\mathcal J}(f) \right)
\end{equation}
where $\mathcal D_\eps$ represents the operator $\mathcal D$ with scaled time~\eqref{time_scaling}. As in \eqref{Boltzmann_eps}, we have
\begin{equation*}
j=\frac{\bar{j}}{\varepsilon}\,,\quad \tilde\chi=\frac{\bar{\tilde{\chi}}}{\varepsilon}\,,\quad R=\frac{\bar{R}}{\varepsilon}\,,
\end{equation*}
in which the latter two quantities give $\mathcal J=\bar{\mathcal J}/\varepsilon$ according to \eqref{Boltz} and~\eqref{isoscattering} and in which $R=\bar{R}/\varepsilon$ is induced by $\phi_0=\bar{\phi}_0/\varepsilon$ and $\phi_1=\bar{\phi}_1/\varepsilon$ in the approximate case~\eqref{coll}. Again, we assume the terms with a bar to be independent of~$\varepsilon$. 

By the time scaling \eqref{time_scaling}, the velocity $v=dx/dt$ of the background matter is scaled as
$$
v=\eps\bar v
$$
which by \eqref{D+} and \eqref{D-} results in a scaling of the operator $\mathcal D$ as
\begin{equation}
\label{D_eps}
\mathcal D_\eps=\eps\mathcal D^++\mathcal D^-
\end{equation}
since the time dependency is completely contained in the denominator of the coefficients in the part $\mathcal D^+$ of the operator $\mathcal D$ which is symmetric with respect to $\mu$.

\begin{remark}
By this time scaling, we consider big reaction and collision rates over long times and small velocities of the background matter, which results in a smaller influence of the aberration and the Doppler effect arising from the motion of the matter described by the part $\mathcal D^+$ of the Boltzmann operator~$\mathcal D$. We do not discuss here whether this additional scaling of time is physically reasonable for supernova neutrinos in the diffusion regime for which we consider it. In order to decide this, we would need a detailed scale analysis of the Boltzmann equation for radiative transfer that should be performed elsewhere. However, the same scaling has already been successfully applied for the Boltzmann equation in gas dynamics~\cite{DegondJin05big}. We also refer to~\cite[Rem.~5]{BardosSantosSentis84}
where for the diffusion approximation of a transport equation with the same scaling structure as \eqref{Boltzmann_eps_teps} with \eqref{D_eps}, the scaling of the mean free path and of the time has been rigorously justified. Analogously scaled kinetic equations are, e.g., considered in \cite{GoudonPoupaud01} and \cite[Sec.~5.1]{Perthame04} with respect to the diffusion limit (without the time scaling one rather obtains the hydrodynamic limit, see, e.g., \cite[Sec.~5.1]{Perthame04}, \cite{DegondEtAl05} {or~\cite{SpeckStrain10}}). Finally, also in our case of radiative transfer, the additional time scaling turns out to be well suited for the derivation of the diffusion source as we will see in the following.
\end{remark}

Inserting the expansion \eqref{f_expanded} for $f$ into \eqref{D_Hilbert_1} with the scaled operator $\mathcal D=\mathcal D_\eps$ leads to the formal series
\begin{equation}
\label{D_Hilbert_2}
\mathcal D_\eps(f) = \mathcal D^-(f_0) +\varepsilon\big(\mathcal D^+(f_0)+\mathcal D^-(f_1)\big) +\varepsilon^2\big(\mathcal D^+(f_1)+\mathcal D^-(f_2)\big) +\dots\,.
\end{equation}

With this series as the left hand side of the time and reaction scaled Boltzmann equation~\eqref{Boltzmann_eps_teps}, we obtain the hierarchy of equations
\begin{eqnarray}
\label{H2Hie1}
0&=&j+\mathcal J(f_0)\,,\\[1mm]
\label{H2Hie2}
\mathcal D^{-}(f_{0}) &=& \bar{\mathcal J}(f_1)\,,\\[1mm]
\label{H2Hie3}
\mathcal D^+(f_{i-2})+\mathcal D^-(f_{i-1}) &=& \bar{\mathcal J}(f_i)\,,\quad\quad i = 2,3, \dots\,,
\end{eqnarray}
that corresponds to $\varepsilon$-powers $i-1$ for $i=0,1,\dots$.

The fact that the components in expansion \eqref{D_Hilbert_2} corresponding to the operator part $\mathcal D^+$ appear in a higher order than the ones corresponding to $\mathcal D^-$ allows us to derive the diffusion-reaction-type approximation of the Boltzmann equation as well as diffusion source up to the order $\BigO(\eps^2)$ without having to neglect $\mathcal D^+\big(\mathcal D^+(\cdot)\big)$ terms since they are now naturally of higher order.

\begin{theorem}[$\Sigma$ in the diffusion limit III] 
\label{weak_theorem_H_time}
Let $\mathcal C(f)$ be given by \eqref{coll} with $\tilde\chi\neq 0$, $\tilde\chi+\phi_0\neq 0$ and $\tilde\chi+\phi_0-\phi_1\neq 0$. We consider the decomposition
$f = f^t + f^s$ inducing $\beta = \beta^t +\beta^s$ with the angular means $\beta^t=\frac{1}{2}\int_{-1}^1 f^t d\mu$ and $\beta^s=\frac{1}{2}\int_{-1}^1 f^s d\mu$. Then the diffusion-reaction-type equation \eqref{diff_react_weak}
is an approximation of the angular mean of Boltzmann's equation \eqref{Boltzmann_eps_teps} of order $\BigO(\varepsilon^2)$. Furthermore, expression \eqref{diff_source}
is an approximation of order $\BigO(\varepsilon^2)$ of the (angular mean of the) diffusion source~$\Sigma$ in~\eqref{eq:trapped}. 
\end{theorem}

The proof reuses arguments by which Lemma~\ref{Liebend_approx} was proved. In fact, it relies entirely on the variant of Lemma~\ref{Liebend_approx} that is based on the Hilbert hierarchy \eqref{H2Hie1}--\eqref{H2Hie3} instead of the Chapman--Enskog hierarchy \eqref{CEHie1} and~\eqref{CEHie2}. This variant is proved in the following.

\begin{proof}[Proof of Theorem \ref{weak_theorem_H_time}]
The first equation \eqref{H2Hie1} of the hierarchy \eqref{H2Hie1}--\eqref{H2Hie3} is the same as the first equation \eqref{CEHie1} in the Chapman--Enskog hierarchy \eqref{CEHie1} and \eqref{CEHie2} considered in Lemma~\ref{Liebend_approx} and, therefore, has the unique solution $f_0=j/\tilde\chi$ for which \eqref{intDf0} holds. 

The second equation \eqref{H2Hie2} is treated as equation \eqref{CEHie2} for $i=1$ with \eqref{Df0} in the proof of Lemma~\ref{Liebend_approx}. We only need to replace $\mathcal{D}$ by $\mathcal{D}^-$ in \eqref{zeroth_f1}--\eqref{Df0-f1} and in the result~\eqref{eps_f1}. The formulas even simplify since by the first equation in \eqref{intDf0} the right hand side in \eqref{zeroth_f1}, i.e., the angular mean of $f_1$, and the second summand in \eqref{eps_f1} vanish so that the latter now reads
\begin{equation}
\label{eps_f1-}
\varepsilon f_1 =\frac{-1}{\tilde\chi +\phi_0}\left(\mathcal D^-(f_0) + 3\mu\phi_1 \lambda\frac{1}{2}\int_{-1}^1\mathcal D^-(f_0)\mu d\mu\right)\,.
\end{equation}
By the proof of Lemma~\ref{Liebend_approx} it becomes clear that with this result the assertions \eqref{intDeps_f1} and \eqref{leading_order} are even true if we replace ``$\doteq$'' by ``$=$'' 
since $\mathcal D^+\big(\mathcal D^+(\cdot)\big)$ terms no longer occur.

To finish the proof we add the angular means of the first three equations in the hierarchy \eqref{H2Hie1}--\eqref{H2Hie3} and, considering again \eqref{collision_vanishes}, obtain
\begin{align*}
\frac{1}{2}\int_{-1}^{1} \big(\mathcal D(f_0) + \mathcal D^-(\varepsilon f_1)\big) d\mu &= j+\frac{\tilde\chi}{2}\int_{-1}^1\left(\sum_{i=0}^2\eps^i f_i\right)d\mu\\[2mm]
&=j+\frac{\tilde\chi}{2}\int_{-1}^1\left(\sum_{i=0}^\infty\eps^i f_i\right)d\mu+\BigO(\varepsilon^2)
=j+\tilde\chi\beta+\BigO(\varepsilon^2)\,,
\end{align*}
taking the formal limit in the series and considering that $j$ and $\tilde\chi$ are of order~$\BigO(\varepsilon^{-1})$. The required diffusion-reaction-type equation \eqref{diff_react_weak} as an approximation of \eqref{Boltzmann_eps_teps} of order $\BigO(\varepsilon^2)$ is obtained by inserting \eqref{intDf0} and \eqref{intDeps_f1} on the left hand side while setting $\beta_0 = \beta^t$ and $\beta=\beta^t+\beta^s$. The corresponding approximation of the diffusion source \eqref{diff_source} follows from the comparison of \eqref{diff_react_weak} with~\eqref{eq:trapped} while taking into account~\eqref{f_t_is_beta_t}.
\end{proof}

For more details concerning the statement on the diffusion source in Theorem~\ref{weak_theorem_H_time} compared to its appearance in the trapped particle equation~\ref{eq:trapped} we refer to Remark~\ref{ftnotneciso} that also applies here.

\begin{remark}
If one treats equation \eqref{H2Hie3} for the order $\eps^2$, i.e.,
\begin{equation}
\label{H_second_order}
\mathcal D^+(f_{1})+\mathcal D^-(f_{2}) = \bar{\mathcal J}(f_2)
\end{equation}
according to \eqref{zeroth_f1}--\eqref{Df0-f1}, one obtains an analogous formula as \eqref{eps_f1} for $\eps f_2$. Concretely, one only needs to replace $\mathcal{D}(f_0)$ by $\mathcal{D}^+(f_0)+\mathcal{D}^-(\eps f_1)$ on the right hand side of~\eqref{eps_f1}. Plugging this into \eqref{H_second_order} and taking into account that $f_1$ in \eqref{eps_f1-} is only made up by $\mathcal D^-$ terms, one can see that there are still no $\mathcal D^+\big(\mathcal D^+(\cdot)\big)$ terms as in \eqref{D++terms} occurring in the second order equation of the hierarchy \eqref{H2Hie1}--\eqref{H2Hie3}. These terms only appear in the equation for the order $\eps^3$ since this one contains the term $\mathcal D^+(f_{2})$ and the formula for $\eps f_2$ is also made up by a $\mathcal{D}^+(f_0)$ term as just seen.\\
\indent 
The reaction-type equation \eqref{react_weak} is not available with the time scaling \eqref{time_scaling} since the angular means of the first two equations \eqref{H2Hie1} and \eqref{H2Hie2} only provide the zero operator on the left hand side. In order to achieve this with the first two equations of the hierarchy, one would need the (inverse) operator scaling $\mathcal D^++\eps \mathcal D^-$ instead of~\eqref{D_eps}. This could be achieved with the inverse time scaling $t=\eps \bar t$ (i.e., short times considered) while scaling the right hand side by $\eps^{-2}$ (i.e., very large reaction and collision rates considered). However, the third equation \eqref{H2Hie3} of the hierarchy would then be made up entirely by $\mathcal D^+\big(\mathcal D^+(\cdot)\big)$ terms on its left hand side.\\
\indent 
Conversely, if we consider very long times, i.e., the time scaling $t=\bar t/\eps^2$ with the reaction and collision scaling as in \eqref{Boltzmann_eps}, then the first $\mathcal{D}^+$ term, in particular the time derivative, only appears in the fourth equation of the hierarchy. As in the proof of Theorem~\ref{weak_theorem_H_time}, we then obtain only the diffusion term in \eqref{intDeps_f1} by the first three equations, i.e., a stationary diffusion equation as an approximation of the corresponding scaled Boltzmann equation of order $\BigO(\eps^2)$. However, if we compare this equation with \eqref{eq:trapped}, we obtain the reaction term in the source term $\Sigma$.\\
\indent
Finally, there is also a version of Theorem \ref{weak_theorem_H_time} based on a Chapman--Enskog expansion and Enskog's postulate \eqref{f_dep_beta} instead of a Hilbert expansion. In this context, the time derivative in the term $\mathcal D^+(f_{i-2})$ in \eqref{H2Hie3} would be replaced by the corresponding order term as in~\eqref{dt_Cauchy}. To establish the result corresponding to Theorem~\ref{weak_theorem_H_time} for the Chapman--Enskog expansion with the same time scaling \eqref{time_scaling}, we need to force that term to be close enough to the time derivative term in $\mathcal D^+(f_{i-2})$ for $i=2$ as done in Lemma~\ref{weak_lemma}. Concretely, we need assertion \eqref{D0CEH} in which $\BigO(\eps)$ is replaced by~$\BigO(\eps^2)$ in order to relate the angular mean of the operator in the Chapman--Enskog expansion with the one in the Hilbert expansion up to the correct order. To guarantee this, the proof of Lemma~\ref{weak_lemma} shows that we need assumption \eqref{eps1ass} in which $\BigO(\eps)$ is replaced by $\BigO(\eps^2)$.
\end{remark}


In order to obtain the free streaming limit we need to consider long times 
while assuming moderate reaction and collision rates. Therefore, we consider the time scaling \eqref{time_scaling}, i.e., the operator scaling \eqref{D_eps}, and no longer impose the scaling on the reaction and the collision terms, thus we deal with the scaled Boltzmann equation
\begin{equation}
\label{Boltzmann_teps}
\eps\mathcal D^+(f)+\mathcal D^-(f)=j+\mathcal{J}(f)\,.
\end{equation}

With this scaling of the Boltzmann equation and the expansion \eqref{D_Hilbert_2} of the left hand side, collecting the same $\varepsilon$-powers, we arrive at the hierarchy of equations
\begin{eqnarray}
\label{H3Hie1}
\mathcal D^{-}(f_{0})&=&j+\mathcal J(f_0)\,,\\[1mm]
\label{H3Hie2}
\mathcal D^+(f_{i-1})+\mathcal D^{-}(f_{i}) &=& \mathcal J(f_i)\,,\quad\quad i = 1,2, \dots\,,
\end{eqnarray}
collecting the terms of the same $\varepsilon$-powers $\eps^i$, $i=0,1,2,\dots$. The first equation \eqref{H3Hie1} of this hierarchy is not time-dependent and turns out to provide the stationary state limit for the streaming component $f^s$ of $f$. Since only the angular mean of \eqref{H3Hie1} is considered, we do not need to assume the special form \eqref{coll} for the collisions since by \eqref{collision_vanishes} they do not play any role here. However, in contrast to \eqref{D-f0} and the derivation of \eqref{intDeps_f1}, we cannot assume isotropy of $f_0$ here. 

\begin{theorem}[$\Sigma$ in the free streaming limit]
\label{Free_thm}
We consider the decomposition $f = f^t + f^s$ inducing $\beta = \beta^t +\beta^s$ with the angular means $\beta^t=\frac{1}{2}\int_{-1}^1 f^t d\mu$ and $\beta^s=\frac{1}{2}\int_{-1}^1 f^s d\mu$ as well as the expansions \eqref{f_expanded} and \eqref{beta_expanded} for $f$ and its angular mean~$\beta$. Then the stationary state equation
\begin{equation}
\label{FreeBol2}
\frac{1}{r^2}\frac{\partial}{\partial r}\left( r^2\frac{1}{2}\int_{-1}^1f^s\mu d\mu \right) = j-\tilde\chi \beta^s
\end{equation}
is an approximation of the angular mean of Boltzmann's equation \eqref{Boltzmann_teps} of order $\BigO(\eps)$ with $f^s=f_0$ and thus $\beta^s = \beta_0$ and $\beta^t = \beta-\beta^s = \BigO(\varepsilon)$. In addition,
\begin{equation*}
\Sigma_{\rm{free}} := j
\end{equation*}
is an approximation of order $\BigO(\eps)$ of the (angular mean of the) diffusion source $\Sigma$ in~\eqref{eq:streaming}. 
\end{theorem}

\begin{proof} 
By \eqref{D-} and \eqref{collision_vanishes}, the angular mean of the first equation \eqref{H3Hie1} in the hierarchy \eqref{H3Hie1} and~\eqref{H3Hie2} gives
\begin{equation*}
\frac{1}{2}\int_{-1}^{1} \left(\mu \frac{\partial f_0}{\partial r}+ \frac{1}{r}(1-\mu^2)\frac{\partial f_0}{\partial \mu} \right) d\mu =j-\frac{\tilde\chi}{2}\int_{-1}^{1} f_0 d\mu\,.
\end{equation*}
Integrating by parts we get
$$
\frac{1}{2}\int_{-1}^{1}\frac{1}{r}(1-\mu^2)\frac{\partial f_0}{\partial \mu}d\mu=\left[\frac{1}{r}(1-\mu^2)f_{0}\right]_{\mu=-1}^{\mu=1}+\frac{1}{2}\int_{-1}^{1}\frac{1}{r}2\mu f_0 d\mu=\frac{2}{r}\frac{1}{2}\int_{-1}^{1} f_0 \mu d\mu\,,
$$
so that, after interchanging integration and differentiation, \eqref{FreeBol2} can be written as
$$
\frac{\partial}{\partial r}\frac{1}{2}\int_{-1}^{1} f_0 \mu d\mu+\frac{2}{r}\frac{1}{2}\int_{-1}^{1} f_0 \mu d\mu=j-\tilde\chi\beta_0\,.
$$
By the product rule, this equation can be more compactly formulated as
\begin{equation*}
\frac{1}{r^2}\frac{\partial}{\partial r}\left( r^2\frac{1}{2}\int_{-1}^1f_0\mu d\mu \right) = j-\tilde\chi \beta_0\,.
\end{equation*}
Now, setting $f^s=f_0$ and thus $\beta^s=\beta_0$ provides equation \eqref{FreeBol2} and the first statement of Theorem~\ref{Free_thm}. The second statement on the source term $\Sigma$ follows from the comparison of~\eqref{FreeBol2} with~\eqref{eq:streaming}. 
\end{proof}

\begin{remark}
The scaling in \eqref{Boltzmann_teps} contains the time scaling \eqref{time_scaling}, i.e., one considers long times, and no reaction and collision scaling. On a physical level, one can argue that one should in fact introduce the scaling 
$$
j=\eps \bar j\,,\quad\tilde\chi=\eps\bar{\tilde{\chi}}\,,\quad R=\eps\bar R\,,\quad \mathcal{J}=\eps\bar{\mathcal{J}}
$$
of reactions and collisions because in the free streaming limit, e.g., far away from the core of a collapsed star, one can assume vanishing reaction and collision rates. Since in the free streaming limit we also assume the distribution function to be in a stationary state, we can then consider an even larger time scale, e.g., the time scaling
$$
t=\frac{\bar t}{\eps^2}\,,
$$
{compare \cite[p.\;131]{CercignaniKremer02}.} The resulting scaled Boltzmann equation then reads
\begin{equation}
\label{Boltzmann_tteps_eps}
\eps^2\mathcal{D}^+(f)+\mathcal{D}^-(f)=\eps j+\eps\mathcal{J}(f)
\end{equation}
and exhibits the hierarchy
\begin{eqnarray*}
\mathcal{D}^{-}(f_{0}) &=& 0\,,\\[1mm]
\mathcal{D}^{-}(f_{1}) &=& \bar j+\bar{\mathcal{J}}(f_0)\,,
\\[1mm]
\mathcal{D}^+(f_{i-2})+\mathcal{D}^{-}(f_{i}) &=& \bar{\mathcal{J}}(f_{i-1})\,,\quad\quad i = 2,3, \dots\,,
\end{eqnarray*}
corresponding to $\eps$-powers $\eps^i$, $i=0,1,\dots$. Considering the sum of the angular means of the first two equations of this hierarchy and the same arguments as in the proof of Theorem~\ref{Free_thm}, one can then derive equation \eqref{FreeBol2} with $f^s=f_0+\eps f_1$ as an order $\BigO(\eps^2)$ approximation of the scaled Boltzmann equation \eqref{Boltzmann_tteps_eps}. Analogously, the source term $j$ is then also determined up to the order $\BigO(\eps^2)$. In fact, one obtains $\BigO(\eps^n)$ approximations for any $n\in\N$ if one considers the scaled Boltzmann equation~\eqref{Boltzmann_tteps_eps} with $\eps^2$ in the time scaling replaced by $\eps^n$ and $\eps$ in the reaction and collision scaling replaced by $\eps^{n-1}$.\\
\indent
Note that we do not require the collision term $\mathcal C(f)$ to be of any specific form here. Since we only take angular means of the equations in the hierarchy, by \eqref{collision_vanishes} it is even irrelevant how we scale the collisions.\\
\indent
Finally, we remark that one can also consider these scalings for the free streaming limit in Chapman--Enskog expansions with Enskog's postulate~\eqref{f_dep_beta}. In fact, since in contrast to the reaction and the diffusion limit, the time-dependency is always of higher order in the free streaming limit, the Chapman--Enskog expansion is identical to the Hilbert expansion up to the order that needs to be considered here. 
\end{remark}

The following table summarizes the scaling results in Theorems \ref{weak_theorem_react_H}, \ref{weak_theorem_H_time} and~\ref{Free_thm}.
\vspace*{2mm}
\begin{center}
\begin{tabular}{|c||c|c|}\hline 
\backslashbox{{reac./coll.~scaled
}\mbox{\hspace*{1cm}}
\vspace*{.5mm}
}
{\vspace*{-2mm}
{time scaled
}}
&yes &no \\[1mm]\hline\hline & & \\[-2mm]
yes &\;diffusion limit\; &\;reaction limit\; \\[1mm]\hline & & \\[-2mm]
no &\;free streaming limit\; &\; full Boltzmann \;\\[1mm]\hline
\end{tabular}
\end{center}
\vspace*{4mm}

With the source terms $\Sigma_{\rm{diff}}$, $\Sigma_{\rm{reac}}$ and $\Sigma_{\rm{free}}$ in the diffusion, reaction and free streaming limit by Theorems \ref{weak_theorem_react_H}, \ref{weak_theorem_H_time} and~\ref{Free_thm}, we can now define a global source term
\begin{eqnarray}
\label{SigmaDef}
\Sigma_{\rm glob} :&=& \min\left\{ \max\Big[\Sigma_{\rm{diff}} \,,\, \Sigma_{\rm{reac}}\Big]\,,\, \Sigma_{\rm{free}}\right\}\nonumber\\[2mm] 
&=& \min\left\{ \max\left[- \frac{1}{r^2}\frac{\partial}{\partial r}\left(r^2\frac{\lambda}{3}\frac{\partial\beta^t}{\partial r}\right) +\tilde\chi \beta^s , \;\tilde\chi \beta^s\right],\, j\,\right\},
\end{eqnarray}
or, with $\tilde{\Sigma}_{\rm{reac}}$ instead of $\Sigma_{\rm{reac}}$, the IDSA source term
\begin{eqnarray}
\label{SigmaDef2}
\tilde\Sigma_{\rm glob} :&=& \min\left\{ \max\Big[\Sigma_{\rm{diff}}\, ,\, \tilde\Sigma_{\rm{reac}}\Big]\,,\, {\Sigma}_{\rm{free}}\right\}\nonumber\\[2mm] 
&=& \min\left\{ \max\left[- \frac{1}{r^2}\frac{\partial}{\partial r}\left(r^2\frac{\lambda}{3}\frac{\partial\beta^t}{\partial r}\right) +\tilde\chi \beta^s ,\; 0\right],\, j\,\right\}\,.
\end{eqnarray}
as in \eqref{Sigma}. However, in contrast to the system of coupled equations \eqref{eq:trapped}, \eqref{eq:streaming} and \eqref{Sigma} of the IDSA, we do not need to assume isotropy $f^t=\beta^t$, i.e.,~\eqref{f_t_is_beta_t}, so that we consider the equations
\begin{equation}
\label{diff_react_weak_Sigma}
\frac{d\beta^t}{cd t}+\frac{1}{3}\frac{d \ln\rho}{cd t}\omega\frac{\partial\beta^t}{\partial \omega} 
=  j-\tilde\chi \beta^t-\tilde\Sigma
\end{equation}
\begin{equation*}
\frac{1}{r^2}\frac{\partial}{\partial r}\left( r^2\frac{1}{2}\int_{-1}^1f^s\mu d\mu \right) = j-\tilde\chi \beta^s
\end{equation*}
given by Theorems \ref{weak_theorem_react_H}, \ref{weak_theorem_H_time} and~\ref{Free_thm}, and coupled by either $\tilde\Sigma=\Sigma_{\rm glob}$ or $\tilde\Sigma=\tilde\Sigma_{\rm glob}$.

We close this investigation by a short motivation why the choice of limiters from above and below for the source term $\Sigma_{\rm{diff}}$ as in the min-max-expressions \eqref{SigmaDef} and \eqref{SigmaDef2} is indeed mathematically reasonable. We base our considerations on the approximation \eqref{diff_react_weak_Sigma} of the angularly integrated Boltzmann equation with as yet undefined but, clearly, isotropic~$\tilde\Sigma$. Recall that we have $\tilde\chi=j+\chi$ by Subsection~\ref{radiative}.

\begin{lemma}
\label{limiter_lemma}
Let the background matter, i.e., $\rho$, the reaction rates $j$ and $\tilde\chi$ and the source term $\tilde\Sigma$ in \eqref{diff_react_weak_Sigma} be in a stationary state limit, i.e., independent of time. Then $\tilde\Sigma$ is limited from above and below by
\begin{equation}
-\chi\leq\tilde\Sigma\leq j\,.
\label{limiter_ab}
\end{equation}
\end{lemma}

\begin{proof}
Since $\rho$ is independent of time, the second term on the left hand side of~\eqref{diff_react_weak_Sigma} vanishes and we obtain the reaction equation
$$
\frac{d \beta^t}{cdt}=j-\tilde\chi \beta^t-\tilde\Sigma
$$
with time-independent data $j$, $\tilde\chi$ and $\tilde\Sigma$ on the right hand side. Consequently, this equation possesses the analytical solution
$$
\beta^t(t)=\beta^t(0) e^{-ct\tilde\chi}+\left(1-e^{-ct\tilde\chi}\right)\frac{j-\tilde\Sigma}{\tilde\chi}\,.
$$
For $t\to\infty$ the solution approaches the stationary state limit 
$$
\beta^t_\infty=\lim_{t\to\infty}\beta^t(t)=\frac{j-\tilde\Sigma}{\tilde\chi}\,.
$$
Since, just as $f^t$, the function $\beta^t(t)=\frac{1}{2}\int_{-1}^1 f^t(t)d\mu$ is a distribution function for every $t>0$, so~is~$\beta^t_\infty$, i.e., we have $0\leq \beta^t_\infty\leq 1$ and, therefore,
\begin{equation}
\label{prelimiters}
0\leq\frac{j-\tilde\Sigma}{\tilde\chi}\leq 1\,.
\end{equation}
With $\tilde\chi=j+\chi$ this estimation is equivalent to~\eqref{limiter_ab}.
\end{proof}

By the stationary state assumption in Lemma \ref{limiter_lemma}, it turns out that the source term $\Sigma_{\rm{free}}=j$ in the free streaming limit is indeed an upper bound for $\tilde\Sigma$ because of the non-negativity of the distribution function. The bound $0\leq\tilde\Sigma$ in the IDSA~\eqref{SigmaDef2} is more restrictive than $-\chi\leq\tilde\Sigma$ in~\eqref{limiter_ab}, and even more so $\tilde\chi\beta^s\leq\tilde\Sigma$ in~\eqref{SigmaDef}. However, $-\chi\leq\tilde\Sigma$ in \eqref{limiter_ab} arises from the upper bound $\beta^t\leq 1$ which one can regard as quite extreme. If instead, one uses the thermal equilibrium function $j/\tilde\chi$ (cf.~\cite[p.\;1177]{LiebendoerferEtAl09big}) as an upper bound for $\beta^t$, i.e., if one substitutes $1$ by $j/\tilde\chi$ in \eqref{prelimiters}, one obtains $0\leq\tilde\Sigma$. 


\section*{Acknowledgments}
The authors thank N.~Vasset for stimulating discussions.


\bibliographystyle{plain}
\bibliography{IDSA}

\end{document}